\documentclass{amsart}

\usepackage{latexsym,amsfonts,amssymb,exscale,enumerate}
\usepackage{amsmath,amsthm,amscd,mathrsfs}
\usepackage{hyperref}

\input xy
\usepackage[all]{xy}
\xyoption{line}
\xyoption{arrow}
\xyoption{color}
\SelectTips{cm}{}

\usepackage{tikz}
\usetikzlibrary{decorations.markings}
\usetikzlibrary{decorations.pathreplacing}

\tikzset{->-/.style={decoration={
  markings,
  mark=at position #1 with {\arrow{>}}},postaction={decorate}}}

\tikzset{middlearrow/.style={
        decoration={markings,
            mark= at position 0.5 with {\arrow{#1}} ,
        },
        postaction={decorate}
    }
}

\usepackage{graphicx}
\usepackage{color}

\newcommand{\onenn}[1]{{\mathbf 1}_{#1}}

\newcommand{\F}{\cal{F}}

\newcommand{\arxiv}[1]{\href{http://arxiv.org/abs/#1}{\tt arXiv:\nolinkurl{#1}}}

\theoremstyle{plain}

\newtheorem{itheorem}{Theorem}

\theoremstyle{definition}

\theoremstyle{remark}

\theoremstyle{definition}
\newtheorem{thm}{Theorem}[section]
\newtheorem{cor}[thm]{Corollary}

\newtheorem{lem}[thm]{Lemma}
\newtheorem{rem}[thm]{Remark}
\newtheorem{prop}[thm]{Proposition}

\newcommand{\U}{\dot{{\bf U}}}
\newcommand{\Ucat}{\cal{U}}

\newcommand{\UA}{{_{\cal{A}}\dot{{\bf U}}}}
\newcommand{\sym}{{\rm Sym}}

\newcommand{\maps}{\colon}

\newcommand{\refequal}[1]{\xy {\ar@{=}^{#1}
(-1,0)*{};(1,0)*{}};
\endxy}

\newcommand{\cat}[1]{\ensuremath{\mbox{\bfseries {\upshape {#1}}}}}

\newcommand{\To}{\Rightarrow}

\newcommand{\Hom}{{\rm Hom}}

\renewcommand{\to}{\rightarrow}
\newcommand{\excise}[1]{}
\newcommand{\hor}{\circ}
\newcommand{\ver}{\cdot}

\def\Id{\mathrm{Id}}

\def\mf{\mathfrak}
\def\shuffle{\,\raise 1pt\hbox{$\scriptscriptstyle\cup{\mskip
               -4mu}\cup$}\,}

\newcommand{\ii}{ \textbf{\textit{i}}}

\numberwithin{equation}{section}

\def\emph#1{{\sl #1\/}}

\let\hat=\widehat
\let\tilde=\widetilde

\let\theta=\vartheta
\let\epsilon=\varepsilon

\usepackage{bbm}
\def\C{{\mathbbm C}}
\def\N{{\mathbbm N}}

\def\Z{{\mathbbm Z}}
\def\Q{{\mathbbm Q}}

\def\cal#1{\mathcal{#1}}%
\def\1{\mathbbm{1}}%
\def\nn{\notag}

\def\la{\langle}
\def\ra{\rangle}

\newcommand{\bbe}[1]{\xybox{%
  (-2,0)*{};
  (2,0)*{};
  (0,0);(0,-18) **\dir{-}; ?(.5)*\dir{<}+(2.3,0)*{\scriptstyle{#1}};
}}

\newcommand{\bbf}[1]{\xybox{%
  (-2,0)*{};
  (2,0)*{};
  (0,0);(0,-18) **\dir{-}; ?(.5)*\dir{>}+(2.3,0)*{\scriptstyle{#1}};
}}

\newcommand{\bbpef}{\xybox{%
  (-6,0)*{};
  (6,0)*{};
  (-4,0)*{}="t1";
  (4,0)*{}="t2";
  "t1";"t2" **\crv{(-4,-6) & (4,-6)}; ?(.15)*\dir{>} ?(.9)*\dir{>};
}}
\newcommand{\bbpfe}{\xybox{%
  (-6,0)*{};
  (6,0)*{};
  (-4,0)*{}="t1";
  (4,0)*{}="t2";
  "t2";"t1" **\crv{(4,-6) & (-4,-6)}; ?(.15)*\dir{>} ?(.9)*\dir{>};
}}

\newcommand{\bbcfe}[1]{\xybox{%
  (-6,0)*{};
  (6,0)*{};
  (-4,0)*{}="t1";
  (4,0)*{}="t2";
  "t1";"t2" **\crv{(-4,6) & (4,6)}; ?(.15)*\dir{>} ?(.9)*\dir{>}
  ?(.5)*\dir{}+(0,2)*{\scriptstyle{#1}};
}}
\newcommand{\bbcef}[1]{\xybox{%
  (-6,0)*{};
  (6,0)*{};
  (-4,0)*{}="t1";
  (4,0)*{}="t2";
  "t2";"t1" **\crv{(4,6) & (-4,6)}; ?(.15)*\dir{>}
  ?(.9)*\dir{>} ?(.5)*\dir{}+(0,2)*{\scriptstyle{#1}};
}}

\newcommand{\lowrru}[1]{\xybox{%
  (-8,0)*{};
  (8,0)*{};
  (-6,-18)*{};(6,-9)*{} **\crv{(-6,-13) & (6,-15)} ?(1)*\dir{>};
  (6,-9)*{};(6,0)*{}  **\dir{-} ?(.3)*\dir{ }+(2,0)*{\scs {\bf j}};
}}

\newcommand{\lowllu}[1]{\xybox{%
  (-8,0)*{};
  (8,0)*{};
  (6,-18)*{};(-6,-9)*{} **\crv{(6,-13) & (-6,-15)} ?(1)*\dir{>};
  (-6,-9)*{};(-6,0)*{}  **\dir{-} ?(.3)*\dir{ }+(-2,0)*{\scs {\bf j}};
}}

\newcommand{\bbdl}[1]{\xybox{%
  (2,0);(0,-8) **\crv{(2,-2)&(0,-6)}; ?(.5)*\dir{>}
}}
\newcommand{\bbdlu}[1]{\xybox{%
  (2,0);(0,-8) **\crv{(2,-2)&(0,-6)}; ?(.5)*\dir{<}
}}
\newcommand{\bbdr}[1]{\xybox{%
  (-2,0);(0,-8) **\crv{(-2,-2)&(0,-6)}; ?(.5)*\dir{>}
}}
\newcommand{\bbdru}[1]{\xybox{%
  (-2,0);(0,-8) **\crv{(-2,-2)&(0,-6)}; ?(.5)*\dir{<}
}}

\DeclareGraphicsExtensions{.pdf,.eps}
\usepackage{psfrag}
\usepackage{bbm}

\def\cal#1{\mathcal{#1}}

\def \k {\mathbbm{k}}
\def \Z {\mathbbm{Z}}

\def \N {\mathbbm{N}}
\def \Q {\mathbbm{Q}}
\def \E {\mathcal{E}}
\def \F {\mathcal{F}}
\def \U {\mathcal{U}}

\def \C {\mathcal{C}}
\def \Tr{\operatorname{Tr}}

\def \Span{\operatorname{Span}}
\def \Ob{\operatorname{Ob}}

\def \HH{\operatorname{HH}}

\def \Id {{\rm Id}}

\def\k{\mathbbm{k}}

\newcommand\nc{\newcommand}
\nc\rnc{\renewcommand}
\nc\Kar{\operatorname{Kar}}
\nc\End{\operatorname{End}}
\nc\modQ {{\mathbb Q}}
\nc\modZ {{\mathbb Z}}
\nc\simeqto{\overset{\simeq}{\longrightarrow }}
\nc\modC {{\mathcal C}}
\nc\modD {{\mathcal D}}
\nc\CC{\mathbf{C}}

\newcommand{\scs}{\scriptstyle}

\nc\calU{\mathcal{U}}
\nc\cU{\calU}
\nc\col{\colon\thinspace}
\nc\calA{\mathcal{A}}
\nc\Ab{\mathbf{Ab}}
\nc\Ko{K_0}
\nc\TrhorCC{\Tr^{\mathrm{hor}}(\CC)}
\nc\AdCat{\mathbf{AdCat}}
\nc\TrCC{\Tr(\CC)}
\nc\Udot{\dot{\mathcal{U}}}
\nc\diag{\mathrm{d}}
\nc\modU {\mathcal{U}}
\nc\bfU{\mathbf{U}}
\nc\dU{\dot{\mathbf U}}
\nc\dUZ{{_\modZ\dot{\mathbf U}}}
\nc\UZ{{_\modZ \mathbf U} }
\nc\fsl{\mathfrak{sl}}
\nc\Uaa{{\bf U} (\mathfrak{sl}_2\otimes \Q[t,t^{-1}])}
\nc\UZslt{{_\modZ\mathbf{U}} (\mathfrak{sl}_2\otimes \Q[t])}
\nc\UdZslt{{_\modZ\dot{\mathbf{U}}} (\mathfrak{sl}_2\otimes \Q[t])}
\nc\LL{L^+\fsl_2}
\nc\UL{\mathbf U(\LL)}
\nc\UZL{\UZ(\LL)}
\nc\dUZL{\dUZ(\LL)}
\nc\dUL{\dU(\LL)}
\nc{\im}{\rm im}
\nc\Kom{\rm Kom}
\nc\GL{\rm{GL}}
\nc\g{\mathfrak{g}}

\nc\tG{\tilde{G}} \nc\tE{\tilde{E}}
\nc\Vect{\rm Vect}
\nc{\Gras}{{\rm {Gr}}}
\nc\FMod{\rm FMod}
\nc\yto[1]{\underset{#1}{\to}}
\nc\Ear{\yto{E}}

\newcommand\sE{{\cal{E}}}
\newcommand\sF{{\cal{F}}}

\def\l{\lambda}
\newcommand{\iccbub}[2]{
\xybox{%
 (-6,0)*{};
  (6,0)*{};
  (-4,0)*{}="t1";
  (4,0)*{}="t2";
  "t2";"t1" **\crv{(4,6) & (-4,6)}; ?(.7)*\dir{}+(-2,0)*{\scs #2}
  ?(.05)*\dir{>} ?(1)*\dir{>};
  "t2";"t1" **\crv{(4,-6) & (-4,-6)};
   ?(.3)*\dir{}+(0,0)*{\bullet}+(0,-3)*{\scs {#1}};
}}
\newcommand{\icbub}[2]{
\xybox{%
 (-6,0)*{};
  (6,0)*{};
  (-4,0)*{}="t1";
  (4,0)*{}="t2";
  "t2";"t1" **\crv{(4,6) & (-4,6)};?(.7)*\dir{}+(-2,0)*{\scs #2};
   ?(0)*\dir{<} ?(.95)*\dir{<};
  "t2";"t1" **\crv{(4,-6) & (-4,-6)};
   ?(.3)*\dir{}+(0,0)*{\bullet}+(0,-3)*{\scs {#1}};
}}

\newcommand{\curg}{\mathscr{C}\!\mf{g}}
\newcommand{\curkg}{\mathscr{C}_{\!\k}\mf{g}}
\newcommand{\cur}[1]{\mathscr{C}_{\! #1}\mf{g}}
\newcommand{\curpg}{\mathscr{C}^+\!\mf{g}}
\newcommand{\curpmg}{\mathscr{C}^\pm\!\mf{g}}
\newcommand{\curmg}{\mathscr{C}^-\!\mf{g}}
\newcommand{\dcurpg}{\dot{\mathscr{C}}^+\!\mf{g}}
\newcommand{\dcurpmg}{\dot{\mathscr{C}}^\pm\!\mf{g}}
\newcommand{\dcurmg}{\dot{\mathscr{C}}^-\!\mf{g}}
\newcommand{\dcurzg}{\dot{\mathscr{C}}^0\!\mf{g}}
\newcommand{\curzg}{\mathscr{C}^0\!\mf{g}}

\newcommand{\dcurkg}{\dot{\mathscr{C}}_{\!\k}\mf{g}}

\allowdisplaybreaks

\begin{document}

\date{\today}

\title{Current algebras and categorified quantum groups}

\author{Anna Beliakova}
\address{Universit\"at Z\"urich, Winterthurerstr. 190
CH-8057 Z\"urich, Switzerland}
\email{anna@math.uzh.ch}

\author{Kazuo Habiro}
\address{Research Institute for Mathematical Sciences, Kyoto University, Kyoto, 606-8502, Japan}
\email{habiro@kurims.kyoto-u.ac.jp}

\author{Aaron D.~Lauda}
\address{University of Southern California, Los Angeles, CA 90089, USA}
\email{lauda@usc.edu}

\author{Ben Webster}
\address{University of Virginia,
  Charlottesville, VA 22903, USA}
\email{bwebster@virginia.edu}

\begin{abstract}
We identify the trace, or 0th Hochschild homology, of type ADE categorified quantum groups with the corresponding current algebra of the same type.  To prove this, we show that 2-representations defined using categories of modules over  cyclotomic (or deformed cyclotomic) quotients of KLR-algebras correspond to local (or global) Weyl modules.  We also investigate the implications for centers of categories in 2-representations of categorified quantum groups.
\end{abstract}

\newcommand{\onel}{{\mathbf 1}_{\lambda}}


\maketitle

\tableofcontents

%
\section{Introduction}

One very powerful idea in mathematics is categorification, and its
necessary partner decategorification.  Most work in recent years has
understood decategorification to mean taking the Grothendieck group,
but there are other ways of interpreting this idea.  The one we will consider in this paper is the notion of {\bf trace}.

The trace $\Tr(\modC )$ of a $\Bbbk$-linear category $\modC$ is $\Bbbk$-vector space given by
\begin{gather*}
  \Tr(\C )=
\left( \bigoplus_{x\in \Ob(\modC )}\End_\C(x)
\right)/\Span_\k\{fg-gf\},
\end{gather*}
where 
$f$ and $g$ run through all pairs of morphisms $f\col x\to y$, $g\col y\to x$ with $x,y\in
\Ob(\modC )$.

Trace (or 0th Hochschild homology) and Grothendieck group are closely
related: there is a {\bf Chern
character map} $$h_{\mathcal{C}}\maps K_0(\mathcal{C})\to\Tr(\mathcal{C})$$
 sending the class of an object to the image of its identity
morphism in the trace.  The interplay of these two kinds of
decategorification has shown up in many contexts, most classically in
the various generalizations of the Riemann-Roch theorem.  See
\cite{CW} for a more categorical perspective and \cite{BGHL} for more
general discussion by the first three authors and Guliyev.

In a certain sense, trace decategorification is richer than
Grothendieck decategorification.  The Chern character map is usually injective,
but often fails to be surjective, so there are many classes in the
trace which do not correspond to any object. To use an extremely loose
analogy, the Grothendieck group is something like $H_0$ of a space,
and the trace or Hochschild homology like its homology.  In fact, the
Borel-Moore homology of a space $X$ can be interpreted as the Hochschild homology of
the category of constructible sheaves on $X$.

We will concentrate on only one small aspect of this large topic.  For
any 2-category $\mathbf{C}$, we can consider its trace, as defined in
Section \ref{sec_traceK0-2cats}. This trace is naturally a category.
A 2-representation of $\mathbf{C}$ is a 2-functor to
$\mathsf{Cat}$\footnote{More generally, we can speak of a
  2-representation of $\mathbf{C}$ in an arbitrary bicategory
  $\mathcal{D}$, such as the bicategory of rings, bimodules, and
  bimodule homomorphisms. This is a 2-functor $\mathbf{C}\to \mathcal{D}$.}, the
2-category of categories, functors and natural transformations.  As explained in Section~\ref{sec:2functor} the 2-representation  $R$ induces a representation of the 1-category $\Tr(\mathbf{C})$ sending each object $c$ to the trace of the category $R(c)$.

Actually, let us specialize this yet further: the 2-category
$\mathbf{C}$ that we'll consider is the categorified quantum group $\Ucat^*=\Ucat^*_Q(\mf{g})$
attached to a symmetric Kac-Moody Lie algebra $\mf{g}$~\cite{KL3}.  This is a 2-category
with many interesting aspects; for us the most important is its
``representation theory.''  For each highest weight $\lambda$, there
are two representation categories $\Ucat^{\lambda,\ast}$, and $\check{\Ucat}^{\lambda,\ast}$, with
Grothendieck groups that agree with the simple highest weight
representation $V(\lambda)$ of $\mf{g}$.  The category $\Ucat^{\lambda,\ast}$ is the
category of projective modules over the cyclotomic quotient as introduced in
\cite[\S 3.4]{KL1} and $\check{\Ucat}^{\lambda}$ the category
of projective modules over the deformed cyclotomic quotient defined in \cite{Web5,RouQH}.

When we take the Grothendieck group of the Karoubi completion of $\Ucat^*$, we obtain the
category $\dot {\bf {U}}(\mf{g})$, which is the idempotented version of the universal enveloping algebra
$\bf{U}(\mf g)$; the trace will prove to be quite a bit larger.
Our study of the trace was motivated by geometric considerations.
For each highest weight $\lambda$, there is collection of quiver varieties, constructed by
Nakajima \cite{Nak98}.
The quiver varieties and 2-category $\Ucat^*$ are closely related; $\Ucat^*$
acts in a natural way on the (quantum) coherent sheaves on these
varieties \cite{CKLquiver,Webcatq}, and many constructions which first
appeared in one context have analogs in the other (for example,
Lusztig's canonical basis appears naturally in both).

This philosophy suggests that the trace of the category $\Ucat^*$
should be connected to the homology of quiver varieties\footnote{The
  reader who is paying attention here might rightly complain
  ``Shouldn't it be connected to the Hochschild homology of $\Ucat^*$?''
Actually, in the geometric context, the grading on $\Ucat^*$
becomes homological in nature, and what we think of as the trace is
really part of the Hochschild homology.}.
Work of Varagnolo shows that
there is an action of
the current algebra $\bf { U}(\mf{g}[t])$ on the whole Borel-Moore
homology (or cohomology) of the quiver varieties for $\lambda$, identifying their sum with the Weyl
module of highest weight $\lambda$ over the current algebra (or dual Weyl module).
In this paper we will discuss the analogue of Varagnolo's construction in our
algebraic context, which is given by the trace decategorification
we have discussed.

\begin{itheorem} \label{thm:introA}
 Assume $\g$ is type ADE.
Then $\Tr(\Ucat^{\lambda,\ast})$ is isomorphic to the local Weyl module of highest
weight $\lambda$, and $\Tr(\check{\Ucat}^{\lambda,\ast})$ is isomorphic to the global
 Weyl module.  Dually, the center of $\Ucat^{\lambda,\ast}$  is isomorphic
 to the dual local Weyl module.
\end{itheorem}

In fact, this result has recently been shown independently by Shan,
Varagnolo, and Vasserot~\cite{SVV}; their techniques are quite similar
to ours, having been inspired by the same geometric considerations.

Our motivation in studying traces was to identify the trace of the 2-category $\Ucat^*$ itself.  Here we prove the following theorem:

\begin{itheorem} \label{thm:introB}
Assume $\g$ is type ADE. The trace $\Tr(\Ucat^*)$ is canonically isomorphic to the
  category $\dot{\bf{ U}}(\mf{g}[t])$.   The isomorphisms of Theorem
  \ref{thm:introA} are
  induced by this isomorphism.
\end{itheorem}
This result intimately links the study of 2-representations of $\Ucat^*(\mf{g})$ with the representation theory of the current algebra $\bf{ U}(\mf{g}[t])$.
 As explained above, any 2-representation of $\Ucat^*$ gives rise to a representation of $\Tr(\Ucat^*)$ and hence the current algebra.

The 2-category $\Ucat^*(\mf{g})$ is  known to act  on numerous categories of interest including:
\begin{itemize}
    \item categories of modules over the symmetric group~\cite{CR,BK1},
    \item parabolic category $\mathcal{O}$ for $\mf{gl}_N$~\cite{BS,Web5,HS},
    \item derived categories of coherent sheaves on Nakajima quiver varieties ~\cite{CKLquiver, Rou2},
  \item coherent sheaves on certain convolution varieties obtained from the affine Grassmannian \cite{Cautis-rigid},
    \item categorified Fock space representations of certain
      Heisenberg algebras \cite{CauLic,Cautis-rigid},
     \item category $\mathcal{O}$ for a rational Cherednik algebra of
       $G(n,1,r)$~\cite{ShanCrystal},
    \item categories of $\mf{sl}_n$-foams used in link homology theory~\cite{Mackaay-foams,LQR,QR}, and
    \item categories of $\mf{sl}_n$-matrix factorizations~\cite{MY}.
\end{itemize}
Theorem~\ref{thm:introB} indicates that all of these 2-representations
give rise to current algebra representations.  As discussed earlier,
there is Chern map relating the Grothendieck and trace
decategorifications, and the induced map commutes with the $\g$-action;
surprisingly, Theorem~\ref{thm:introA} shows that that there are
2-representations with the same Grothendieck decategorification can
have different trace decategorifications.

Finally, another natural construction on categories is the notion of the center of an additive category $Z(\mathcal{C})$.  This is defined as the commutative monoid of endo-natural transformations of the identity functor $\End(\Id_{\mathcal{C}})$.    The center and trace of a category are closely related.   Here we also prove the following:
\begin{itheorem} \label{thm:center}
For $\mathfrak{g}$ of any type, any 2-representation of $\Ucat^{\ast}$ into the 2-category of additive $\Bbbk$-linear categories gives rise to an action of $\Tr(\Ucat^{\ast})$ on the centers of the categories defining the 2-representation.
\end{itheorem}

In particular, in type ADE, we can identify $\Tr(\Ucat^{\ast})$ with the current algebra to give an action of it on the centers of all of the categories listed above.  A special case of this fact was already observed by Brundan.  He made the surprising discovery that one could define an action of the Lie algebra $\hat{\mf{g}}:=\mf{gl}_{\infty}(\mathbb{C})$ on the center $Z(\cal{O}) = \bigoplus_{\nu} Z(\cal{O}_{\nu})$ of all integral blocks $\cal{O}_{\nu}$ of category $\cal{O}$ for $\mf{gl}_n$~\cite{Brundan}.
In this action, the Chevalley generators of $\hat{\mf{g}}$ act as certain trace maps associated to canonical adjunction maps between special translation functors that arise from tensoring with a $\mf{g}$-module and its dual.
Theorem \ref{thm:center} gives a new construction of Brundan's action as well as extends it to an action of the current algebra associated to $\hat{\mf{g}}$.

The paper is organized as follows: In the Sections
\ref{trace}--\ref{categorified-quantum}, we present some general facts
about the trace and define different versions of the  categorified
quantum groups and current algebra.  In Section \ref{current-algebra},
we define the map of the current algebra to the trace.
Finally, we prove Theorem A, using results from Theorems
\ref{local-isomorphism} and \ref{global-isomorphism} and Corollary \ref{center-weyl}. Theorem B is equivalent to Theorem \ref{thm:graded}. Theorem C is proven in the last
section where rescaling 2-functors needed to make $\Ucat^*(\mf{g})$ cyclic are also studied.

\vspace{0.2in}

{\bf Acknowledgments.}
The authors are grateful to Vyjayanthi  Chari for helpful discussions about representations of current algebras,
and to the anonymous referee for pointing out the difference between the current algebra and the universal enveloping algebra $U(\mf{g}[t])$ in infinite type.  We are also grateful to Jun Hu for comments on an earlier version of this article.
 A.B. was  supported by Swiss National Science Foundation under Grant PDFMP2-141752/1.
K.H. was supported by JSPS Grant-in-Aid for Scientific Research (C) 24540077. A.D.L  was partially supported by NSF grant DMS-1255334 and by the John Templeton Foundation.  B.W was supported by the NSF under Grant DMS-1151473.

\section{The trace decategorification map}
\label{trace}
%

%
\subsection{Traces of linear categories}
%

In what follows, we recall the notion of the trace of linear
categories and generalizations.  For more details, see
e.g. \cite{BHLZ,BGHL}.  We focus on linear categories over a fixed
field $\k$.

Let $\modC $ be a small $\k$-linear category.  Thus, the hom spaces
$\modC(x,y)$ are $\k$-vector spaces, and the composition is bilinear
over $\k$.

Define the {\em trace} $\Tr(\modC )$ of $\modC$, also known as the
zeroth Hochschild-Mitchell homology $\HH_0(\modC)$, by
\begin{gather*}
  \Tr(\C )
  =\left(\bigoplus_{x\in \Ob(\modC )}\End_\modC (x)\right)/\Span_\k\{fg-gf\},
\end{gather*}
where $f$ and $g$ runs through all pairs of morphisms $f\col
x\to y$, $g\col y\to x$ with
$x,y\in\Ob(\modC)$.  (Note that $\Tr(\C)$ depends on the base field
$\k$, but we usually omit it in the notation.)  For $f\col x\to x$,
let $[f]\in \Tr(\modC )$ denote the corresponding equivalence class.

Recall that a $\k$-linear category with one object is identified with a
$\k$-algebra.  For a $\k$-algebra $A$, we set
\begin{gather*}
  \Tr(A)=\HH_0(A) = A/[A,A]= A/\Span_\k\{ab-ba\;|\;a,b\in A\}.
\end{gather*}

The trace $\Tr$ gives a functor from the small $\k$-linear categories to
the $\k$-vector spaces.  If $F\col\mathcal C\to\mathcal D$ is a
$\k$-linear functor, then $F$ induces a linear map on traces
\begin{gather*}
  \Tr(F)\col\Tr(\mathcal C)\to\Tr(\mathcal D)
\end{gather*}
given by
$\Tr(F)([f]) = [F(f)]$
for endomorphisms $f\col x\to x$ in $\mathcal C$.  Furthermore, if
$\alpha\maps F \To F\maps \mathcal{C} \to \mathcal{D}$ is a natural
transformation of $\k$-linear functors, then $\alpha$ gives rise to a
linear map
\begin{align} \label{eq:nat}
  \Tr(\alpha)\col\Tr(\mathcal C) &\to\Tr(\mathcal D) \\
  [f \maps x \to x] & \mapsto [\alpha_x \ver F(f)\maps F(x)\to F(x)]. \nn
\end{align}

\begin{lem}
  \label{r1}
  Let $\C$ be a $\k$-linear additive category.
  Let $S\subset\Ob(\C)$ be a subset such that every object in $\C$ is
  isomorphic to the direct sum of finitely many copies of objects in
  $S$.  Let $\C|_S$ denote the full subcategory of $\C$ with
  $\Ob(\C|_S)=S$.  Then, the inclusion functor $\C|_S\to \C$ induces
  an isomorphism
  \begin{gather}
    \label{e9}
    \Tr(\modC|_S)\cong\Tr(\modC)
  \end{gather}
\end{lem}

\begin{proof}
  The inclusion functor $\C|_S\to\C$ factors, uniquely up to natural
  isomorphisms, as
  \begin{gather*}
    \C|_S \to (\C|_S)^\oplus \simeq \C,
  \end{gather*}
  where $(\C|_S)^\oplus$ is the additive closure of $\C|_S$.  Here
  $\C|_S \to (\C|_S)^\oplus$ is the canonical functor, and $(\C|_S)^\oplus
  \simeq \C$ is the canonical equivalence.
  These functors induce isomorphisms on $\Tr$
  \begin{gather*}
    \Tr(\C|_S) \cong \Tr((\C|_S)^\oplus) \cong\Tr(\C).
  \end{gather*}
\end{proof}

%
\subsection{Split Grothendieck groups and Chern character}
\label{sec:split-groth-groups}
%

For a $\k$-linear additive category $\modC $, the {\em split
Grothendieck group} $K_0(\modC )$ of $\modC $ is the abelian group
generated by the isomorphism classes of objects of $\modC $ with
relations $[x\oplus y]_{\cong}=[x]_{\cong}+[y]_{\cong}$ for
$x,y\in\Ob(\modC)$.  Here $[x]_{\cong}$ denotes the isomorphism class
of~$x$.

For a $\k$-linear additive category $\modC$, the {\em Chern character}
for $\C$ is the $\k$-linear map
\begin{gather*}
  h_\modC \col K^\k_0(\C):=K_0(\modC )\otimes\k\longrightarrow \Tr(\modC )
\end{gather*}
defined by $h_\modC ([x]_{\cong})=[1_x]$ for $x\in \Ob(\modC )$.
(Although $h_\modC$ can be defined on $K_0(\modC)$, we consider
only the above $\k$-linear version for simplicity.)

%
\subsection{Chern character for Krull-Schmidt categories}
\label{sec:chern-character-maps}
%

A $\k$-linear additive category $\modC$ is said to be {\em
Krull-Schmidt} if every object in $\modC$ decomposes in a unique way as the direct
sum of finitely many indecomposable objects with local endomorphism
rings.  In a Krull-Schmidt category,
\begin{itemize}
\item an object is indecomposable if and only if its endomorphism ring
is local,
\item idempotents split.
\end{itemize}

Let $\modC$ be a $\k$-linear Krull-Schmidt category.  Fix a subset
$I\subset \Ob(\modC)$ consisting of exactly one from each isomorphism
class of indecomposable objects in $\modC$.  Then the split
Grothendieck group $K_0(\modC)$ is a free abelian group with basis
given by the isomorphism classes of indecomposable objects in $\modC$.
Hence we have
\begin{gather}
  \label{e7}
  K^\k_0(\modC) \cong \k\cdot I =\bigoplus_{x\in I}\k.
\end{gather}

Let $\mathcal{J}$ be the two-sided ideal in $\C|_I$ generated by
\begin{itemize}
\item $J(\C(x,x))$ for $x\in I$, and
\item $\C(x,y)$ for $x,y\in I$, $x\neq y$.
\end{itemize}
where $J(\C(x,x))$ is the Jacobson radical of the endomorphism ring
$\C(x,x)$.
That is, $\mathcal{J}$ is the smallest family of $\k$-subspaces
$\mathcal{J}(x,y)$ for $x,y\in I$ which contains the subspaces given
above and is closed under left and right composition with morphisms in
$\C|_I$.  This is {\it not} a subcategory, since it does not contain
the identity morphisms of objects.

In fact, $\mathcal{J}$ coincides with the Jacobson radical of the linear
category $\C|_I$, defined in \cite{Kelly,Mitchell}.

\begin{lem}
  \label{r13}
  For $x\in I$, we have
  \begin{gather}
    \label{e11}
    \mathcal{J}(x,x) = J(\C(x,x)),
  \end{gather}
\end{lem}

\begin{proof}
  We have
  \begin{gather*}
    \mathcal{J}(x,x)
    = J(\C(x,x))+\sum_{y\in I,y\neq x}\C(y,x)\circ\C(x,y).
  \end{gather*}
  Therefore, it suffices to show that, for $x,y\in I$ with $x\neq y$, we have
  $\C(y,x)\circ\C(x,y)\subset J(\C(x,x))$.

  We will show that if $f\col x\to y$, $g\col y\to x$, then $gf\in
  J(\C(x,x))$.  Suppose $gf\not\in J(\C(x,x))$ for contradiction.
  Since $\C(x,x)$ is local, it follows that $gf$ is an isomorphism.
  Hence $f(gf)^{-1}g\in\C(y,y)$ is an idempotent.  Since $\C$ is
  Krull-Schmidt, it follows that $y$ has a direct summand isomorphic
  to $x$.  Since $y$ is indecomposable, it follows that $y\cong x$.  This is a
  contradiction.
\end{proof}

By Lemma \ref{r13}, the quotient category $(\C|_I)/\mathcal{J}$ has the
following hom spaces.
\begin{gather}
  \label{e12}
  ((\C|_I)/\mathcal{J})(x,y) = \begin{cases}
    D_x:=\C(x,x)/J(\C(x,x)),& x=y,\\
    0, & x\neq y.
  \end{cases}
\end{gather}
Note that $D_x$ is a division algebra over $\k$.  By \eqref{e12}, we
have
\begin{gather*}
  \Tr((\C|_I)/\mathcal{J})
  \underset{}{\cong}
  \bigoplus_{x\in I}D_x/[D_x,D_x].
\end{gather*}

Let $\eta\col \bigoplus_{x\in I}\k\to \bigoplus_{x\in I}D_x$ be the
composite
\begin{gather}
  \label{e10}
  \bigoplus_{x\in I}\k
  \underset{\eqref{e7}}{\cong}
  K^\k_0(\C)
  \overset{h_\C}{\longrightarrow}
  \Tr(\C)
  \underset{\eqref{e9}}{\cong}
  \Tr(\C|_I)
  \overset{}{\twoheadrightarrow}
  \Tr((\C|_I)/\mathcal{J})
  \underset{}{\cong}
  \bigoplus_{x\in I}D_x/[D_x,D_x],
\end{gather}
where $\twoheadrightarrow$ is induced by the projection
$\C|_I\to(\C|_I)/\mathcal{J}$.
It is easy to check that $\eta=\bigoplus_{x\in I}\eta_x$, where
$\eta_x\col\k\to D_x/[D_x,D_x]$ is defined by $\eta_x(1_\k)=[1_{D_x}]$.
By \eqref{e10}, we have the following.

\begin{lem}
  \label{r6}
  If $\eta_x$ is injective (i.e. $1\not\in[D_x,D_x]$) for all $x\in
  I$, then $h_\C$ is injective.
\end{lem}

Using Lemma \ref{r6}, we prove the following.

\begin{prop}\label{lem:Chern-injective}
  Let $\k$ be a perfect field.  Let $\modC$ be a $\k$-linear
  Krull-Schmidt category with finite dimensional endomorphism algebra
  for each indecomposable object.  Then, the Chern character map
  $h_\C\colon K_0(\modC)\otimes_\Z\k\to\Tr(\modC) $ is injective.
\end{prop}

\begin{proof}
  By Lemma \ref{r6}, it suffices to prove that for each $x\in I$,
  $1_{D_x}\not\in[D_x,D_x]$.

  Note that $D_x$ is a finite dimensional division algebra over $\k$.
  Let $K$ be the center of $D_x$, which is a finite extension field of
  $\k$.

  For $u\in D_x$, let $L_u\col D_x\to D_x$ be left multiplication by
  $u$, which is a $K$-linear map.  Define a $K$-linear map $\tau_x\col D_x\to K$ by
  \begin{gather*}
    \tau_x(u) =
    \operatorname{tr}(L_u).
  \end{gather*}
  We have $\tau_x([D_x,D_x])=0$, since
  \begin{gather*}
    \operatorname{tr}(L_{[u,v]}) =\operatorname{tr}([L_u,L_v])=0.
  \end{gather*}
  We have
  \begin{gather*}
    \tau_x(1_{D_x})=
    \operatorname{tr}(L_{1_{D_x}})
    =\dim_K D_x.
  \end{gather*}
  If $\k$ is of characteristic $p>0$, then since $\k$ is perfect, it
  follows from \cite[Theorem 13]{AAA} that $\dim_K D_x$ is not divisible by $p$.
  Hence it follows that $1_{D_x}\not\in[D_x,D_x]$, regardless of the
  characteristic of $\k$.
\end{proof}

\begin{prop}
  \label{r16}
  Let $\k$ be a field, and let $\C$ be a $\k$-linear Krull-Schmidt
  category such that for each indecomposable object $x$ we have
  $\C(x,x)/J(\C(x,x))\cong\k$.  (This condition holds when $\k$ is
  algebraically closed and each $\C(x,x)$ is finite dimensional.)
  Then the Chern character $h_\C$ is split injective with a unique
  splitting
  \begin{gather*}
    p_\C\col\Tr(\C)\to K^\k_0(\C)
  \end{gather*}
  such that, for $x\in I$, $p_\C([1_x])=[x]_{\cong}$ and $p_\C([f])=0$
  for $f\in J(\C(x,x))$.
\end{prop}

\begin{proof}
  We have $D_x=D_x/[D_x,D_x]\cong\k$ for each $x\in I$.  Using
  \eqref{e10} and \eqref{e7}, we obtain the result.
\end{proof}

%
\subsection{Graded categories}
%

A {\em graded} $\Bbbk$-linear category is a $\k$-linear category $\C$
equipped with an auto-equivalence
$\la1\ra\col\C\overset{\simeq}{\longrightarrow}\C$.  For $t\ge0$, set
$\la t\ra=\la1\ra^t$, and, for $t<0$, $\la t\ra=\la-1\ra^{-t}$, where
$\la-1\ra\col\C\overset{\simeq}{\longrightarrow}\C$ is an inverse
(unique up to natural isomorphism) of $\la1\ra$.


The auto-equivalence $\la1\ra$ induces $\k$-linear automorphisms
\begin{gather*}
  q=K^\k_0(\la1\ra)\col K^\k_0(\C)\overset{\cong}{\to} K^\k_0(\C),\\
  q=\Tr(\la1\ra)\col \Tr(\C)\overset{\cong}{\to}\Tr(\C),
\end{gather*}
which give $K^\k_0(\C)$ and $\Tr(\C)$, respectively,
$\k[q,q^{-1}]$-module structures.
\excise{Given a graded $\k$-linear category $\modC$, its Grothendieck group
$K_0(\modC)$ is a $\Z[q,q^{-1}]$-module with $q^t[x]:=[x \la t
\ra]_{\cong}$.  Likewise, the trace $\Tr(\modC)$ is
$\k[q,q^{-1}]$-module with $q^t [f]:=[f\la t \ra] $ for an
endomorphism $f \maps x \to x$ in $\modC$.}
The Chern character map
\[
h_\C\col K^\k_0(\modC) \to \Tr(\modC)
\]
is a $\k[q,q^{-1}]$-module homomorphism, since $h_\C$ is a natural
transformation.

A {\em translation} in a graded $\k$-linear category $\C$ is a family
of natural
isomorphisms
\begin{gather*}
  x\overset{\cong}{\to}x\la1\ra.
\end{gather*}
If a graded $\k$-linear category $\C$ admits a translation, then it
makes the action $q$ on $K^\k_0(\C)$ and $\Tr(\C)$ trivial.  Thus, in
this case, $K^\k_0(\C)$ and $\Tr(\C)$ are $\k$-vector spaces rather
than $\k[q,q^{-1}]$-modules.

Given any graded $\k$-linear category $\modC$, we can form a graded
$\k$-linear category $\modC^{\ast}$ with translation, such that
$\Ob(\C)=\Ob(\C^{\ast})$ and
\[
 \modC^{\ast}(x,y) := \bigoplus_{t \in \Z} \modC(x,y\la t\ra),
\]
for all $x,y \in \Ob(\modC)$.  Note that $\C^\ast$ is equipped with a
$\Z$-grading with the degree $t$ hom space given by
\begin{gather*}
  \modC^{\ast}_t(x,y) := \modC(x,y\la t\ra),\quad t\in \Z.
\end{gather*}
\excise{The translations $x \to x \la t \ra$ is given by $1_x$ in $\modC^{\ast}$ since
\[
 1_x \in \modC(x,x\la 0 \ra) = \modC\left(x,\left(x\la t \ra\right) \la -t\ra\right) \subset \modC^*(x,x\la t \ra),
\]
has inverse
\[
 1_x\la t\ra \in \modC(x \la t \ra, x \la t \ra) = \modC(x \la t \ra, (x \la 0 \ra) \la t \ra)   \subset \modC^{\ast}(x\la t\ra, x).
\]
}
Thus $\C^\ast$ is enriched over $\Z$-graded vector spaces.  Hence, the
trace $\Tr(\C^\ast)$ has a $\Z$-graded $\k$-vector space structure
\begin{gather*}
  \Tr(\C^\ast)=\bigoplus_{t\in \Z}\Tr_t(\C^\ast),
\end{gather*}
where $\Tr_t(\C^\ast)$ is spanned by $[f]$ for endomorphisms in
$\C^\ast$ of degree $t$. If $\Tr_t(\C^\ast)$ is nonzero for $t>0$, then the Chern character map is not surjective, since its image is contained in $\Tr_0(\C^\ast)$.

%

\subsection{$K_0$, $\Tr$ and $\Kar$ for 2-categories} \label{sec_traceK0-2cats}
%

We can extend many of the constructions defined above for (additive)
$\Bbbk$-linear categories to the 2-categorical setting.  

A {\em $\Bbbk$-linear 2-category } is a 2-category $\CC$ such that
\begin{enumerate}
\item for $x,y\in \Ob(\CC)$, the category $\CC(x,y)$ is equipped with a
  structure of a $\Bbbk$-linear category,
\item for $x,y,z\in \Ob(\CC)$, the functor
  $\circ\col \CC(y,z)\times \CC(x,y)\rightarrow \CC(x,z)$, called horizontal composition,
    is ``bilinear'' in the sense
  that the functors $-\circ f\col \CC(y,z)\rightarrow \CC(x,z)$ for
  $f\in \Ob(\CC(x,y))$ and $g\circ -\col \CC(x,y)\rightarrow \CC(x,z)$ for
  $g\in \Ob(\CC(y,z))$ are $\Bbbk$-linear functors.
\end{enumerate}
The composition in the categories $\CC(x,y)$ is called vertical and denoted by $\ver$.

The following definitions extend the split Grothendieck group, trace
and Karoubi envelope to the 2-categorical setting. Let $\CC$ be a $\k$-linear 2-category.
\begin{itemize}
\item We define the {\bf split Grothendieck group} $K_0^\k(\CC)$ of $\CC$ to be the $\k$-linear category with $\Ob(K_0^\k(\CC)) = \Ob(\CC)$ and with $K_0^\k(\CC)(x,y) := K_0^\k(\CC(x,y))$
for any two objects $x,y \in \Ob(\CC)$.
For $[f]_{\cong} \in
\Ob(K_0^\k(\CC)(x,y))$ and $[g]_{\cong} \in \Ob(K_0^\k(\CC)(y,z)) $ the
composition in $K_0^\k(\CC)$ is defined by $[g]_{\cong}  [f]_{\cong}
:= [g \circ  f]_{\cong}$.

\item We define the {\bf trace} $\Tr(\CC)$ of $\CC$ to be the $\k$-linear
  category with
  $\Ob(\Tr(\CC))=\Ob(\CC)$ as follows.  For $x,y\in \Ob(\CC)$, set
  $\Tr(\CC)(x,y)=\Tr(\CC(x,y))$.  For $x,y,z\in \Ob(\CC)$, given
  morphisms in this category $[\sigma] \in \Tr({\CC(x,y)})$ and
  $[\tau] \in \Tr({\CC(y,z)})$ represented by 2-morphisms $\sigma$ and
  $\tau$ in $\CC$, we have a product given by
  $[\tau ][\sigma ]=[\tau \hor\sigma ]$.  The identity
  morphism for $x\in \Ob(\Tr(\CC))=\Ob(\CC)$ is given by $[1_{I_x}]$.
  Since $\CC$ is $\k$-linear, $\Tr(\CC)$ obtains an induced linear structure.

\item We define the {\bf Karoubi envelope} $\Kar(\CC)$ of $\CC$ to be the $\k$-linear 2-category with $\Ob(\Kar(\CC)) = \Ob(\CC)$ and with hom categories $\Kar(\CC)(x,y) := \Kar(\CC(x,y))$.  The composition functor
$\Kar(\CC)(y,z) \times \Kar(\CC)(x,y) \to \Kar(\CC)(x,z)$ is induced
by the universal property of the Karoubi envelope from the composition
functor in $\CC$.  The fully-faithful additive functors $\CC(x,y) \to
\Kar(\CC(x,y))$ combine to form an additive $2$-functor $\CC \to
\Kar(\CC)$ that is universal with respect to splitting idempotents in
the  Hom-categories $\CC(x,y)$.
\end{itemize}

The homomorphisms $h_{\CC(x,y)}$ taken over all objects $x,y \in \Ob(\CC)$ give rise to a functor \begin{equation} \label{eq_chernchar}
h_{\CC} \col K_0^\k (\CC) \to \Tr(\CC)
\end{equation}
which is the identity map on objects and sends $K_0^\k(\CC)(x,y)  \to \Tr(\CC)(x,y)$ via the homomorphism $h_{\CC(x,y)}$.  It is easy to see that this assignment preserves composition since
\begin{align}
  h_{\CC}([g]_{\cong} [f]_{\cong}) = h_{\CC}([g \circ  f]_{\cong}) = [1_{g \circ f}]
   = [1_g \hor 1_f] = [1_g]  [1_f] = h_{\CC}([g]_{\cong}) h_{\CC}([f]_{\cong}).
\end{align}

\subsection{$2$-functoriality of $\Tr$ on linear $2$-categories} \label{sec:2functor}

A (strict) $2$-functor $F\col\mathbf C\to \mathbf D$ between linear $2$-categories $\mathbf C$ and
$\mathbf D$ is a {\em linear $2$-functor} if for $x,y\in\Ob(\mathbf C)$ the functor
$F\col\mathbf C(x,y)\to\mathbf D(x,y)$ is linear.  Then $F$ induces a linear functor
\begin{gather*}
  \Tr(F)\col\Tr(\mathbf C)\to \Tr(\mathbf D)
\end{gather*}
such that the map $F\col\Ob(\mathbf C)\to \Ob(\mathbf D)$ on objects gives the map
\begin{gather*}
  \Tr(F)=F\col\Ob(\Tr(\mathbf C))\to \Ob(\Tr(\mathbf D)),
\end{gather*}
and, for $x,y\in \Ob(\mathbf C)$, the linear functor
$F_{x,y}\col\mathbf C(x,y)\to \mathbf D(F(x),F(y))$ induces the linear map
\begin{gather*}
  \Tr(F)_{x,y}=\Tr(F_{x,y})\col\Tr(\mathbf C)(x,y)\to \Tr(\mathbf D)(F(x),F(y)).
\end{gather*}
It is possible to work more generally in the context of linear bicategories and non-strict 2-functors, however this generality is not needed here.

In the case when $\mathbf{D}=\cat{LinCat}$, the 2-category of $\Bbbk$-linear categories, $\Bbbk$-linear functors, and $\Bbbk$-linear natural transformations,   a 2-functor $F\maps \mathbf{C} \to \cat{LinCat}$ can be used to define a representation
\begin{equation}
  \rho_F\maps  \Tr(\mathbf{C}) \to \cat{Vect}_{\Bbbk}
\end{equation}
sending each object $x$ of $\Tr(\mathbf{C})$ to the $\Bbbk$-vector
space $\Tr(F(x))$, and each morphism $[\sigma]\maps x\to y$ in $\Tr(
\mathbf{C})$, with $\sigma \maps f
\To f \maps x \to y$ in $\mathbf{C}$, to the linear map
\begin{gather}
  \label{e1}
  \rho_F([\sigma])\maps \Tr(F(x))\longrightarrow \Tr(F(y)),
\end{gather}
such that for $[g\maps u\to u]\in\Tr(F(x))$ we have
\begin{gather*}
  \rho_F([\sigma])([g])
  =[F(f)(g)\ver F(\sigma)_u]
  =[F(\sigma)_u\ver F(f)(g)].
\end{gather*}
Here, note that $F(\sigma) \maps
F(f) \To F(f) \maps F(x) \to F(y)$ is a natural transformation.
Hence, using equation~\eqref{eq:nat}
we see that \eqref{e1} is well-defined.

%
\section{The current algebra  $\curg$}
\label{sec:current-algebra}
%

%
\subsection{The quantum group $\mathbf{U}_q(\mf{g})$}
%

%
\subsubsection*{Cartan data}
%

For most of this article, we will restrict our attention to Lie algebras of type ADE.  Since there are a few of our results which apply more generally, in this section, we'll discuss more general simply-laced Kac-Moody algebras. However, in all theorems for the rest of the paper {\bf we assume that $\mf g$ is finite type ADE unless we explicitly state otherwise}.  These algebras are associated to a symmetric Cartan datum consisting of
\begin{itemize}
\item a free $\Z$-module $X$ (the weight lattice),
\item for $i \in I$ ($I$ is an indexing set) there are elements $\alpha_i \in X$ (simple roots) and $\Lambda_i \in X$ (fundamental weights),
\item for $i \in I$ an element $h_i \in X^\vee = \Hom_{\Z}(X,\Z)$ (simple coroots),
\item a bilinear form $(\cdot,\cdot )$ on $X$.
\end{itemize}
Write $\langle \cdot, \cdot \rangle \maps X^{\vee} \times X
\to \Z$ for the canonical pairing. These data should satisfy:
\begin{itemize}
\item $(\alpha_i, \alpha_i) = 2$ for any $i\in I$,
\item $\la i,\lambda\ra :=\langle h_i, \lambda \rangle =  (\alpha_i,\lambda)$
  for $i \in I$ and $\lambda \in X$,
\item $(\alpha_i,\alpha_j) \in \{ 0, -1\}$  for $i,j\in I$ with $i \neq j$,
\item $\langle h_j, \Lambda_i \rangle =\delta_{ij}$ for all $i,j \in I$.
\end{itemize}
Hence $(a_{ij})_{i,j\in I}$ is a symmetric generalized Cartan matrix, where $a_{ij}=\langle
h_i, \alpha_j \rangle=(\alpha_i, \alpha_j)$.
We denote by $X^+ \subset X$ the dominant weights which are of the form $\sum_i \lambda_i \Lambda_i$ where $\lambda_i \ge 0$.

Associated to a symmetric Cartan data is a graph $\Gamma$ without loops or multiple edges.  The vertices of $\Gamma$ are the elements of the set $I$ and there is an edge from vertex $i$ to vertex $j$ if and only if $(\alpha_i, \alpha_j) =-1$.
\medskip

The quantum group ${\bf U}={\bf U}_q(\mf{g})$ associated to a simply-laced root datum as above is the unital
associative $\Q(q)$-algebra given by generators $E_i$, $F_i$, $K_{\mu}$ for $i
\in I$ and $\mu \in X^{\vee}$, subject to the relations:
\begin{center}
\begin{enumerate}[i)]
 \item $K_0=1$, $K_{\mu}K_{\mu'}=K_{\mu+\mu'}$ for all $\mu,\mu' \in X^{\vee}$,
 \item $K_{\mu}E_i = q^{\la \mu,i\ra}E_iK_{\mu}$ for all $i \in I$, $\mu \in
 X^{\vee}$,
 \item $K_{\mu}F_i = q^{-\la \mu, i\ra}F_iK_{\mu}$ for all $i \in I$, $\mu \in
 X^{\vee}$,
 \item $E_iF_j - F_jE_i = \delta_{ij}
 \frac{K_i-K_i^{-1}}{q-q^{-1}}$, where $K_i=K_{\alpha_i}$,
 \item For all $i\neq j$ $$\sum_{a+b=-\la i, j \ra +1}(-1)^{a} E_i^{(a)}E_jE_i^{(b)} = 0
 \qquad {\rm and} \qquad
 \sum_{a+b=-\la i, j \ra +1}(-1)^{a} F_i^{(a)}F_jF_i^{(b)} = 0 ,$$
       where $E_i^{(a)}=E_i^a/[a]!$, $F_i^{(a)}=F_i^a/[a]!$, with $[a]!=\prod_{m=1}^a\frac{q^m-q^{-m}}{q-q^{-1}}$.
\end{enumerate} \end{center}

We are primarily interested in the idempotent form $\dot{{\bf U}}_q(\mf{g})$ of ${\bf U}_q(\mf{g})$.

The $\Q(q)$-linear category $\dot{{\bf U}}=\dot{{\bf U}}_q(\mf{g})$
is defined as follows. The objects of $\dot{{\bf U}}$ are elements of $X$.
Given $\l, \nu \in X$, the hom space is defined as the $\Q$-module
$$\dot{{\bf U}}(\l, \nu):=\dot{{\bf U}}/\left( \sum_{\mu\in X^{\vee}} \dot{{\bf U}} (K_\mu-q^{\la\mu,\l\ra})+ \sum_{\mu\in X^{\vee}}(K_\mu-q^{\la\mu,\nu\ra})  \dot{{\bf U}}\right).
$$
The identity morphism of $\l\in X$ is denoted by $1_\l$.
The element in
$ \dot{\bfU}(\l,\mu)$ represented by $x \in {\bfU}$
can be written as
$1_\mu x1_\l=1_\mu x=x1_\l$, where $\mu-\l=|x|$, and
$$ E_i1_{\lambda} = 1_{\lambda+\alpha_i}E_i, \qquad F_i1_{\lambda} = 1_{\lambda-\alpha_i}F_i\ .$$

The composition
in $\dot{\bfU}$ is induced by multiplication in the  algebra, i.e.
$$ (1_\mu x1_\nu) (1_\nu y 1_\l)=1_\mu xy 1_\l$$
for $x,y \in \bfU$, $\l,\mu,\nu \in X$, which is zero unless
 $|x|=\mu-\nu$, $|y|=\nu-\l$.


The integral version $\UA$ is defined as
a $\Z$-linear subcategory of $\dot{{\bf U}}$
whose hom spaces are
generated by products of divided powers
$E_i^{(a)}1_{\lambda}$ and $F_i^{(a)}1_{\lambda}$.

%
\subsection{Definition of the current algebra  $\curg$}
%

First, assume that $\k$ is a field of characteristic 0.
Consider the current algebra
$\curkg$, the associative algebra generated over $\k$ by
$x^+_{i,r}$, $x^-_{i,s}$ and $\xi_{i,k}$ for $r,s,k\in \N\cup\{0\}$ and $i \in I$,
modulo the following relations:
\begin{itemize}
  \item For $i,j \in I$ and $r,s \in \N\cup\{0\}$
\[
[ \xi_{i,r}, \xi_{j,s} ] = 0 \tag{\bf C1}\label{C1}
\]

  \item For $i,j \in I$ and $r \in \N\cup\{0\}$
\[
 [ \xi_{i,0}, x^{\pm}_{j,r} ] =  \pm a_{ij} x^{\pm}_{j,r}\tag{\bf C2}\label{C2}
\]

  \item For $i,j \in I$ and $r \in \N$, $s\in \N\cup\{0\}$
\[
 [\xi_{i,r}, x^{\pm}_{j,s} ] =  \pm a_{ij} x^{\pm}_{j, r+s}\tag{\bf C3}\label{C3}
\]

  \item For $i,j \in I$ and $r,s \in \N\cup\{0\}$
\[
[x^\pm_{i, r+1}, x^\pm_{j,s}]=[x^\pm_{i,r}, x^\pm_{j,s+1}]\tag{\bf C4}\label{C4}
\]
  \item For $i,j \in I$ and $r,s \in \N\cup\{0\}$
\[
 [x^+_{i,r}, x^-_{j,s}] = \delta_{i,j} \xi_{i, r+s}\tag{\bf C5}\label{C5}
\]
  \item
    Let $i \neq j$.  If $a_{ij}=0$, then for $r,s\in\N\cup\{0\}$
\[
[x^{\pm}_{i,r},x^{\pm}_{j,s}]=0. \tag{\bf C6a}\label{C6a}
\]
    If $a_{ij}=-1$, then for $r_1,r_2,s\in\N\cup\{0\}$
\[
[x^{\pm}_{i,r_1},[x^{\pm}_{i,r_2},x^{\pm}_{j,s}]]=0. \tag{\bf C6b}\label{C6b}
\]
\end{itemize}
\noindent
We define
$$| x^\pm_{i,j}|=\pm \alpha_i, \quad |\xi_{j,s}|=0\, .$$

Instead of  \eqref{C3}, some authors use the relation:
for any $i,j \in I$ and  $r,s\in \N\cup\{0\}$
\[
 [\xi_{i,r+1}, x^{\pm}_{j,s} ] =[\xi_{i,r}, x^{\pm}_{j,s+1} ], \tag{\bf C3'}\label{C3prime}
\]
which together with \eqref{C2} implies \eqref{C3}.

In finite type, this algebra has a more familiar realization:
\begin{thm}[Drinfeld \cite{D}]\label{Drinfeld}
  For $\k$ of characteristic 0, we have an isomorphism $\curg\cong \mathbf{U}_\k(\mf g[t])$ to the universal
  enveloping algebra of polynomial currents in $\mf g$ via the map $x^\pm_{i,s}:=x^\pm_i\otimes t^s$, $\xi_{i,k}:=\xi_i\otimes t^k$  for $i\in I$, where
$x^+_i, x^-_i$ and $\xi_i$  are the standard Chevalley generators of $\bfU(\g)$.
\end{thm}
The current algebra is closely connected to the Yangian, which can be
thought of as its quantized universal enveloping algebra (see, for example
\cite{Bag}).

Note that this isomorphism does not hold in infinite type.
Unfortunately, the authors have been unable to find this fact noted in
the literature, but an explicit calculation for
$\g=\mathfrak{\widehat{sl}}_3$ readily confirms it; we thank the
referee for pointing this out.  This is one of
several reasons that the results of Section \ref{sec-injectivity}
cannot be proven outside finite type.  It seems entirely possible that
additional relations are needed to obtain the ``correct'' algebra for
infinite type Cartan data.

For a field $\k$ of characteristic $p$, we should use a divided power
version of the current algebra.  Consider the subalgebra $\cur{\Z}$
be the subalgebra of $\cur{\Q}$  generated over $\Z$ by $(x^{\pm}_{i,a})^r/r!$.  For a general field
$\k$, we let $\curkg\cong \cur{\Z}\otimes_\Z \k$.
We will typically leave out the $\k$ in the notation as understood.

\subsubsection*{Triangular decomposition}

Let $\curpg$, $\curmg$ and $\curzg$ be
the subalgebras of $\curg$ generated by
$\{ (x^+_{i,r})^n/n! \mid i\in I, r\in \N\cup \{0\}\}$, $\{ (x^-_{i,r})^n/n! \mid i\in I, r\in \N\cup \{0\}\}$
and $\{ \xi_{i,r} \mid i\in I, r\in \N\cup \{0\}\}$, respectively.
In finite type, the triangular decomposition of $\mf g$ shows that every element $f \in  \curg$ can be
expressed as a sum
\[ f=\sum f^+ f^0 f^-\quad{\text{where}}\quad f^\pm\in \curpmg , f^0\in \curzg .
\]
%
\subsection{The idempotent form}
%
The idempotented version  $\dot\curg$ of the current algebra is a
$\k$-linear category, whose objects are  $\lambda \in X$.
For $\l, \mu\in X$, the $\k$-vector space of morphisms from $\l$ to $\mu$
is defined as follows:
$$\dot\curg(\l, \mu):= \curg/ I_\xi$$
where
$$I_\xi:=
\sum_{i\in I} \curg\left(\xi_{i,0}-\la i,\l\ra\right) + \sum_{i\in I}\left(\xi_{i,0}-\la i,\mu\ra\right)\curg\, .$$
\noindent
We will denote the identity morphism of $\l\in X$ in $\dot\curg(\l, \l)$
 by $1_\l$. The element in
$ \dot\curg(\l,\mu)$ represented by $x \in \curg$
can be written as
$1_\mu x1_\l=1_\mu x=x1_\l$, $\mu-\l=|x|$. The composition
in $\dot\curg$ is induced by multiplication in the current algebra, i.e.
$$ (1_\mu x1_\nu) (1_\nu y 1_\l)=1_\mu xy 1_\l$$
for $x,y \in \curg$, $\l,\mu,\nu \in X$, which is zero unless
 $|x|=\mu-\nu$, $|y|=\nu-\l$.

\subsubsection*{Triangular decomposition}
In this subsection, we only consider $\g$ of finite type.  Let $\dcurpg$ and $\dcurmg$  be linear
subcategories of $\dot\curg$
whose hom spaces between $\l$ and $\mu$ are
$$ 1_\mu \curpg 1_\l:=\{1_\mu x^+ 1_\l \mid x^+\in \curpg \}$$
and
$$ 1_\mu \curmg 1_\l:=\{ 1_\mu x^-1_\l \mid x^-\in \curmg\}\, ,$$ respectively.
Let
$$\dcurzg := \oplus_\l 1_\l \curzg 1_\l.$$
Then any morphism $f$ of $ \dot\curg$ decomposes as
\[ f=\sum f^+ f^0 f^-\quad{\text{where}}\quad f^\pm\in \dcurpmg , f^0\in \dcurzg .
\]

\subsubsection*{Grading}

Both $\dot\curg$ and $\curg$ are naturally graded.
We will take the convention that for $X\in \g$, we have that $X\otimes
t^m$ has degree $2m$.

\subsubsection*{Shifting}
\label{sec:shifting}

For each $\xi\in \k$, the loop algebra is equipped with an
automorphism $\tau_\xi(X\otimes t^m)=X\otimes (t-\xi)^m$ for any $X\in
\g$.  For any module $M$ over $\curg$, we can precompose its
action with this automorphism to obtain a new module $M_\xi$, which
we will call the {\bf shift} of $M$ by $\xi$.

%
\subsection{Weyl modules}
\label{sec:Weyl}
 
For a fixed $\lambda$, let $m_i=\la i,\lambda\ra$.  Recall that the
universal enveloping algebra ${\bf U}(\g)$ has a integrable representation called the {\bf (finite) Weyl module} $V(\lambda)$.
 We
add the word ``finite'' here to avoid any confusion with the
corresponding modules over the current algebra.  These are modules
generated ${\bf U}(\g)$ by a single vector $v_\lambda$ with defining relations:
\begin{equation}
\g^+ v_{\lambda}=0, \quad \xi_{i}v_\lambda=\la i,\lambda\ra v_\lambda ,\quad
(x^-_{i})^{(m_i+1)} v_\lambda=0\,\qquad\qquad \text{for any $i\in I$.}\label{eq:0}
\end{equation}
If $\k$ has characteristic 0, then these modules give a complete,
irredundant list of the finite dimensional simple modules over ${\bf
  U}(\g)$.  If $\k$ has positive characteristic, then for most $V(\l)$
these will have proper submodules, and the finite dimensional simple
modules are their unique simple quotients.  Let us note for
completeness that we can also define a {\bf quantized Weyl module} $V_q(\l)$
over ${\bf U}_q(\g)$; this is defined by the same equations
\eqref{eq:0}, with the divided power replaced by a quantum divided
power.

Now, we discuss analogs of these modules over the current algebra.  The {\bf global Weyl module} $\mathbb W(\lambda)$ is the $\g[t]$-module
generated over $\curg$ by an element $w_{\lambda}$ with defining relations:
\begin{equation}
\g^+[t] w_{\lambda}=0, \quad \xi_{i,0}w_\lambda=\la i,\lambda\ra w_\lambda ,\quad
(x^-_{i,0})^{m_i+1} w_\lambda=0\,\qquad\qquad\text{for any $i\in I$.}\label{eq:1}
\end{equation}

The ring $\bf {U}({\mathfrak h}[t])$ (which can be thought of as a polynomial
ring in infinitely many variables) has a right action on $\mathbb
W(\lambda)$ sending $uw_\lambda\cdot h=uhw_{\lambda}$.  This action is
not faithful, but rather factors through a finitely generated quotient
$\mathbb{A}_\lambda$.  By
\cite[6.1]{CFK}, the ring $\mathbb{A}_\lambda $ is a
polynomial ring generated by an alphabet
$\{\mathsf{x}_{i,1},\dots,\mathsf{x}_{i,\la i,\lambda\ra}\}_{i\in
    I}$ with $\mathsf{x}_{i,k}$ having degree $2k$.  In particular, its Hilbert
series is $\prod_{i\in I} (1-t)^{-1}\cdots
(1-t^{\la i,\lambda\ra})^{-1}$.  Note that a maximal ideal in
$\mathbb{A}_\lambda $ is naturally encoded by scalars $\nu_{i,k}$
given by the image of $\mathsf{x}_{i,k}$; we will usually consider these
as polynomials \[\boldsymbol{\nu}_i(-z)=z^{\la i,\lambda\ra}+
\nu_{i,1} z^{\la i,\lambda\ra-1}+\cdots +\nu_{i,\la i,\lambda\ra}\]

For a Lie algebra ${\mathfrak a}$, let us denote by ${\mathfrak a}t[t]$ the ideal
of ${\mathfrak a}[t]$ generated by the elements of the form $x\otimes
t^n$ with $x\in {\mathfrak a}$ and $n>0$.

The {\bf local Weyl module} $W(\lambda)$ is the $\g[t]$-module
generated by an element $w_{\lambda}$ with defining relations:

\begin{equation}
\g^+[t] w_{\lambda}=0, \quad {\mathfrak h}t[t]w_\lambda =0,  \quad \xi_{i,0}w_\lambda=\la i,\lambda\ra w_\lambda ,\quad
(x^-_{i,0})^{m_i+1} w_\lambda=0\,\qquad\qquad \text{for any $i\in I$.}\label{eq:2}
\end{equation}
We can also consider the shifts of these modules by scalars
$W_\xi(\lambda),\mathbb W_\xi(\lambda)$; we will call these {\bf shifted
  Weyl modules}.  These arise naturally in the structure theory of
these modules, since:
\begin{lem}[\mbox{\cite[5.8]{CFK}}]
  The specialization of $\mathbb W(\lambda)$ at the maximal ideal for $\boldsymbol{\nu}_i$ in
  $\mathbb{A}_\lambda $ is isomorphic (after possible finite base
  extension) to the tensor product
  $\bigotimes_{\xi}W_{\xi}(\lambda_\xi)$ where $\xi$ ranges over the
  roots $\boldsymbol{\nu}_i(\xi)=0$ for all $i$, and $\lambda_\xi$ are
  roots such that $\la i,\lambda_\xi\ra$ is the multiplicity of
  $\xi$ as a root of $\boldsymbol{\nu}_i(z)$, and
  $\sum_\xi\lambda_\xi=\lambda$.
\end{lem}
This decomposition is unique: $\lambda_\xi$ is the sum of the fundamental
weights with coefficients given by $\xi$'s multiplicities as roots of $\boldsymbol{\nu}_i(z)$.

Both of the global and local Weyl modules are naturally graded with the generating vector
having degree 0,  since the relations \eqref{eq:1} and \eqref{eq:2}
are homogeneous.

The global (resp. local) Weyl modules have a natural universal
property: there is a homomorphism of $W(\lambda)$ (resp. $\mathbb
W(\lambda)$) to an integrable module $M$ sending $w_\lambda\to m\in M$ if
and only if $\g^+[t] m=0$ and $\xi_{i,0}m=\la i,\lambda\ra m$
(resp. also  $ {\mathfrak h}t[t]m=0$).  In particular, this map
will be surjective if $m$ generates $M$, and homogeneous of degree $k$
if $M$ is graded with $m$ of degree $k$.

\subsection{Evaluation modules}
\label{sec:evaluation-modules}

For every $\chi\in A$ in some $\k$-algebra $A$, we have an evaluation homomorphism
$\operatorname{ev}_\chi\colon \g[t]\to \g\otimes A$ sending $x\otimes
t^i\mapsto \chi^ix$ for $x\in \g$.  For any representation $Z$ of
$\g\otimes A$, we have an induced pullback representation $Z_\chi$ over
$\g[t]$.  Particularly interesting cases include:
\begin{itemize}
\item $A=\k$.  In this case, if $V$ is an irreducible representation  over $\g$, then
  $V_\chi$ will also be irreducible.
\item $A=\k[x]$.  In this case, we have the universal evaluation
  module $Z_x$.
\end{itemize}
Note that the shift by $\xi\in\k$ of an evaluation module for $\chi\in
A$ is again an evaluation
module with parameter $\chi+\xi$.  Thus, our notations for shift and
evaluation will not conflict.

Consider the evaluation of finite Weyl module $V_\chi(\lambda)$ when $A=\k$.  Since the highest weight vector in $V_\chi(\lambda)$
satisfies the equations \eqref{eq:1}, we have a surjective map
$\mathbb{W}(\lambda)\to V_\chi(\lambda)$.

  More generally, assume that $\chi_1,\dots, \chi_N$ are distinct scalars.
\begin{lem}
 We have a
  surjective map
  \[\mathbb{W}(\lambda_1+\cdots+\lambda_N)\to
  V_{\chi_1}(\lambda_1)\otimes \dots \otimes V_{\chi_N}(\lambda_N),\]
  sending \[w_{\lambda_1+\cdots +\lambda_N}\mapsto v_{\lambda_1}\otimes
  \cdots \otimes v_{\lambda_N}.\]
\end{lem}
\begin{proof}
The existence of this map is clear from the universal property.
To show that
this map is surjective as well,  note that by Lagrange interpolation,
there exists a
polynomial  $f_i$ such that $f_i(\chi_j)=\delta_{ij}$. In this case, $X\otimes f_i(t)$
acts on $V_{\chi_1}(\lambda_1)\otimes \dots \otimes
V_{\chi_N}(\lambda_N)$ by $1\otimes \cdots\otimes X\otimes
\cdots\otimes 1$, that is, by $X$ in the $i$th tensor factor.  Since
the tensor product of highest weight vectors generates under the
action of these operators, we are done.
\end{proof}
Note that the map factors through $W_{\chi_1}(\lambda_1)\otimes \dots \otimes
W_{\chi_N}(\lambda_N)$.
\begin{lem}\label{lem:coprod-ind}
  For any finite collection of linearly independent elements $u_i\in {\bf{U}}(\g[t])$, there is an
  integer $N$ such that for generic $\chi_1,\dots, \chi_N$ the images $z_i=(\operatorname{ev}_{\chi_1}\otimes
  \cdots\otimes \operatorname{ev}_{\chi_N}) \Delta^{(N)}(u_i)$ under
 the $N$-fold
  coproduct with the universal evaluation in $N$ independent
  parameters remain linearly independent.
\end{lem}
\begin{proof}
To show this result for generic $\chi_1,\dots, \chi_N$ is the same as
to show it for the universal evaluation over $\k[x_1,\dots, x_n]$.
  We can assume that each $u_i$ is a PBW monomial without loss of
  generality.  In this case, we let $N$ be the maximal length of one
  of these monomials.
That is, we consider $u_i=(X_{a_1}\otimes
   t^{b_1}) (X_{a_2}\otimes t^{b_2})\cdots (X_{a_n}\otimes t^{b_n})$
   where $\{X_1,\dots, X_d\}$ is a basis of $\g$.
   The $N$-fold coproduct $\Delta^{(N)}(u_i)$ is of the form \[(X_{a_1}\otimes
   t^{b_1}) \otimes (X_{a_2}\otimes t^{b_2})\otimes \cdots \otimes
   (X_{a_n}\otimes t^{b_n})\otimes 1\otimes \cdots \otimes 1+\cdots
   .\]  Note that this term does not appear in the coproduct of any other
   PBW monomial.

   Now, applying the evaluation, we have the form
\[z_i=x_1^{b_1} \cdots x_n^{b_n}X_{a_1}
   \otimes X_{a_2}\otimes \cdots \otimes
   X_{a_n}\otimes 1\otimes \cdots \otimes 1+\cdots\]
Again, this term will not show up in any other PBW monomial. This
shows the desired independence.
\end{proof}
The following lemma is presumably standard, but we include a short
proof for completeness.
\begin{lem}\label{lem:Weyl-nontrivial}
  For any element $u\in {\bf{U}}_{\k} (\mathfrak{g})$, there is a tensor
  product of two Weyl modules for $\g$ on which it acts non-trivially.
\end{lem}
\begin{proof}
  This is a straightforward consequence of Peter-Weyl if $\k$ has
  characteristic 0.  However, let us give an argument that works in
  arbitrary characteristic.  By PBW, we can write $u=\sum
  u_i^0u^+_iu^-_i$ where $u_i^\pm\in \mathbf{U}^\pm _{\k} (\g)$.  Now, consider the
  action of $u$ on the tensor product of the highest and lowest weight
  vectors $v^+\otimes v^-\in V(\lambda_+)\otimes V(\lambda_-)$ for
some $\lambda_+$ and $\lambda_-$.  We can assume without loss of
generality that these elements are weight vectors, and that we have
used a minimal number of terms subject to this restriction.

Since all elements of $\mathbf{U}^\pm _{\k} (\g)$ kill $v_\pm$, we have that
\[u(  v^+\otimes v^-)=\sum u_i^0u^+_i u^-_i(v^+\otimes v^-)=\sum
u_i^0(u^-_iv^+\otimes u^+_iv^-+\cdots)\]
where the remaining terms have higher weight in the left term and
lower in the right term.

For any linearly independent subset $\{w_i^\pm\}$ of
$\mathbf{U}^\pm(\g)$, the set $\{w_i^\pm v^\pm\}$ is linearly
independent for $\lambda_\pm \gg 0$.  Thus, for  $\lambda_\pm \gg 0$, the
terms of minimal weight in the left term and maximal in the right term
give a linear combination $\sum
u_i^0(u^+_iv^+\otimes u^-_iv^-)=0$.  Since these are weight vectors,
we have obtained a linear dependence in the set $\{u_i^+u_i^-\}$; we
can use this to reduce the number of terms in the sum of $u$,
obtaining a contradiction to the assumption that we had taken the
minimal number possible.
\end{proof}

\begin{lem} \label{lem:nokill}
  No element of ${\bf{U}}_{\k}(\mathfrak{g}[t])\cong \curkg$ kills all global Weyl modules.
\end{lem}
\begin{proof}
We must show that no $u\in {\bf { U}}_{\k} (\g[t])$ can act trivially on all global Weyl modules.
  By Lemma~\ref{lem:coprod-ind}, for generic $\chi_i$, we have $v=(\operatorname{ev}_{\chi_1}\otimes
  \cdots\otimes \operatorname{ev}_{\chi_N})\Delta^{(N)}(u)\neq 0$. Such a
  set of $\chi_i$ must exist if $\k$ is infinite (and we can replace
  $\k$ with an infinite extension without changing the result).

  Thus, we have an algebra map ${\bf { U}}_{\k} (\g[t])\to {\bf { U}}_{\k} (\g)^{\otimes N}\cong
 {\bf { U}}_{\k} (\g^{\oplus N})$ which does not kill $u$.
Applying Lemma \ref{lem:Weyl-nontrivial} to ${\bf { U}}_{\k} (\g^{\oplus
  N})$, we have a tensor product
\[(V(\lambda_{1,1})\otimes \cdots \otimes V(\lambda_{1,k_1}))\boxtimes \cdots
\boxtimes (V(\lambda_{N,1})\otimes \cdots \otimes V(\lambda_{N,k_N}))\]
on which $v$ acts non-trivially.  Here the symbol  $\boxtimes$ is used for the outer tensor product, giving an action of $\g^{\oplus N}$ on the tensor product of $N$ modules over $\g$, and the standard symbol $\otimes$ refers to the internal tensor product in the category of  $\g$-modules.

 That is to say, $u$ acts
non-trivially on \[(V_{\chi_1}(\lambda_{1,1})\otimes \cdots \otimes V_{\chi_1}(\lambda_{1,k_1}))\otimes \cdots
\otimes (V_{\chi_N}(\lambda_{N,1})\otimes \cdots \otimes
V_{\chi_N}(\lambda_{N,k_N})).\]
Thus, necessarily, this shows that $u$ acts non-trivially on
$V_{x_{1,1}}(\lambda_{1,1})\otimes \cdots \otimes
V_{x_{N,k_N}}(\lambda_{N,k_N}),$ for formal parameters $x_{*,*}$, and
thus also when $x_{i,k}$ is replaced by a generic numerical parameter
$\chi_{i,k}\in\k$.

Thus $u$ must act
non-trivially on $V_{\chi_{1,1}}(\lambda_{1,1})\otimes \cdots \otimes
V_{\chi_{N,k_N}}(\lambda_{N,k_N})$.  Since this is a quotient of the global Weyl
module $\mathbb{W}(\lambda_{1,1}+\cdots +\lambda_{N,k_N})$, the action on this
module must be non-trivial as well.
\end{proof}

%

%
%

%
\section{Categorified quantum groups}
\label{categorified-quantum}
%
In this section we
briefly recall
the definition of the 2-category
$\Ucat_Q(\mf{g})$ and list some of its relations.
For the full definition of this 2-category
categorifying
 $\mathbf{U}(\mathfrak{g})$
  we refer to \cite{BHLW}. In \cite{BHLW} this 2-category is called
  $\Ucat^{cyc}_Q(\mf{g})$  and is shown to be isomorphic to
  the one defined earlier in \cite{CLau}, but to have  better
  duality properties. We remove here the index
  $\it{cyc}$ for simplicity.

  For an elementary introduction to the categorification of $\mf{sl}_2$ see \cite{Lau3}.

%
%

%
\subsection{2-category  $\Ucat_Q(\mf{g})$}
%

By a graded linear 2-category we mean a category enriched in graded linear categories, so that the hom spaces form graded linear categories, and the composition map is grading preserving.

Associated to a symmetric Cartan datum define a {\em choice of scalars $Q$} consisting of:
\begin{itemize}
  \item $\left\{ t_{ij}  \mid \text{ for all $i,j \in I$} \right\}$,
\end{itemize}
such that
\begin{itemize}
\item $t_{ii}=1$ for all $i \in I$ and $t_{ij} \in \Bbbk^{\times}$ for $i\neq j$,
 \item $t_{ij}=t_{ji}$ when $a_{ij}=0$.
\end{itemize}

Given a fixed choice of scalars $Q$ we define
the 2-category $\Ucat_Q(\mf{g})$ as a graded linear
2-category consisting of:
\begin{itemize}
\item objects $\lambda$ for $\lambda \in X$.
\item 1-morphisms are formal direct sums of (shifts of) compositions of
$$\onel, \quad \onenn{\l+\alpha_i} \sE_i = \onenn{\l+\alpha_i} \sE_i\onel = \sE_i \onel, \quad \text{ and }\quad \onenn{\lambda-\alpha_i} \sF_i = \onenn{\lambda-\alpha_i} \sF_i\onel = \sF_i\onel$$
for $i \in I$ and $\l \in X$.
In particular, any morphism can be written as a finite formal sum of symbols $\sE_{\ii}\onel\la t\ra$ where $\ii = (\pm i_1, \dots, \pm i_m)$ is a signed sequence of simple roots, $t$ is a grading shift,  $\sE_{+i}\onel := \sE_i\onel$ and $\sE_{-i}\onel:=\sF_i\onel$, and $\sE_{\ii}\onel\la t\ra := \sE_{\pm i_1} \dots \sE_{\pm i_m}\onel \la t\ra$.

\item 2-morphisms are $\Bbbk$-vector spaces spanned by compositions of (decorated) tangle-like diagrams illustrated below.

\begin{align}
  \xy 0;/r.17pc/:
 (0,7);(0,-7); **\dir{-} ?(.75)*\dir{>};
 (0,0)*{\bullet};
 (7,3)*{ \scs \lambda};
 (-9,3)*{\scs  \lambda+\alpha_i};
 (-2.5,-6)*{\scs i};
 (-10,0)*{};(10,0)*{};
 \endxy &\maps \cal{E}_i\onel \to \cal{E}_i\onel\la (\alpha_i,\alpha_i) \ra  & \quad
 &
    \xy 0;/r.17pc/:
 (0,7);(0,-7); **\dir{-} ?(.75)*\dir{<};
 (0,0)*{\bullet};
 (7,3)*{ \scs \lambda};
 (-9,3)*{\scs  \lambda-\alpha_i};
 (-2.5,-6)*{\scs i};
 (-10,0)*{};(10,0)*{};
 \endxy\maps \cal{F}_i\onel \to \cal{F}_i\onel\la (\alpha_i,\alpha_i) \ra  \nn \\
   & & & \nn \\
   \xy 0;/r.17pc/:
  (0,0)*{\xybox{
    (-4,-4)*{};(4,4)*{} **\crv{(-4,-1) & (4,1)}?(1)*\dir{>} ;
    (4,-4)*{};(-4,4)*{} **\crv{(4,-1) & (-4,1)}?(1)*\dir{>};
    (-5.5,-3)*{\scs i};
     (5.5,-3)*{\scs j};
     (9,1)*{\scs  \lambda};
     (-10,0)*{};(10,0)*{};
     }};
  \endxy \;\;&\maps \cal{E}_i\cal{E}_j\onel  \to \cal{E}_j\cal{E}_i\onel\la - (\alpha_i,\alpha_j) \ra  &
  &
   \xy 0;/r.17pc/:
  (0,0)*{\xybox{
    (-4,4)*{};(4,-4)*{} **\crv{(-4,1) & (4,-1)}?(1)*\dir{>} ;
    (4,4)*{};(-4,-4)*{} **\crv{(4,1) & (-4,-1)}?(1)*\dir{>};
    (-6.5,-3)*{\scs i};
     (6.5,-3)*{\scs j};
     (9,1)*{\scs  \lambda};
     (-10,0)*{};(10,0)*{};
     }};
  \endxy\;\; \maps \cal{F}_i\cal{F}_j\onel  \to \cal{F}_j\cal{F}_i\onel\la - (\alpha_i,\alpha_j) \ra  \nn \\
  & & & \nn \\
     \xy 0;/r.17pc/:
    (0,-3)*{\bbpef{i}};
    (8,-5)*{\scs  \lambda};
    (-10,0)*{};(10,0)*{};
    \endxy &\maps \onel  \to \cal{F}_i\cal{E}_i\onel\la 1 + (\l, \alpha_i) \ra   &
    &
   \xy 0;/r.17pc/:
    (0,-3)*{\bbpfe{i}};
    (8,-5)*{\scs \lambda};
    (-10,0)*{};(10,0)*{};
    \endxy \maps \onel  \to\cal{E}_i\cal{F}_i\onel\la 1 - (\l, \alpha_i) \ra  \nn \\
      & & & \nn \\
  \xy 0;/r.17pc/:
    (0,0)*{\bbcef{i}};
    (8,4)*{\scs  \lambda};
    (-10,0)*{};(10,0)*{};
    \endxy & \maps \cal{F}_i\cal{E}_i\onel \to\onel\la 1 + (\l, \alpha_i) \ra  &
    &
 \xy 0;/r.17pc/:
    (0,0)*{\bbcfe{i}};
    (8,4)*{\scs  \lambda};
    (-10,0)*{};(10,0)*{};
    \endxy \maps\cal{E}_i\cal{F}_i\onel  \to\onel\la 1 - (\l, \alpha_i) \ra \nn
\end{align}
\end{itemize}

The vertical composition of morphisms $ab$ is given by stacking $a$ on
top of $b$, and the horizontal composition $a\hor b$ is given by
placing $a$ to the left of $b$ in the plane.

For the whole list of relations on 2-morphisms
we refer to
\cite{BHLW}.  We are listing here only those which will be of importance later.

The 1-morphisms $\cal{E}_i\onel$ and $\cal{F}_i\onel$
are biadjoint. This defines a pivotal structure
 on the 2-category and can be expressed diagrammatically as follows:

\begin{equation}\label{eq_biadjoint2}
 \xy   0;/r.17pc/:
    (8,0)*{}="1";
    (0,0)*{}="2";
    (-8,0)*{}="3";
    (8,-10);"1" **\dir{-};
    "1";"2" **\crv{(8,8) & (0,8)} ?(0)*\dir{>} ?(1)*\dir{>};
    "2";"3" **\crv{(0,-8) & (-8,-8)}?(1)*\dir{>};
    "3"; (-8,10) **\dir{-};
    (12,9)*{\lambda};
    (-5,-9)*{\lambda+\alpha_i};
    \endxy
    \; =
    \;
    \xy 0;/r.17pc/:
    (8,0)*{}="1";
    (0,0)*{}="2";
    (-8,0)*{}="3";
    (0,-10);(0,10)**\dir{-} ?(.5)*\dir{>};
    (5,-8)*{\lambda};
    (-9,-8)*{\lambda+\alpha_i};
    \endxy
\qquad \quad \xy   0;/r.17pc/:
    (-8,0)*{}="1";
    (0,0)*{}="2";
    (8,0)*{}="3";
    (-8,-10);"1" **\dir{-};
    "1";"2" **\crv{(-8,8) & (0,8)} ?(0)*\dir{<} ?(1)*\dir{<};
    "2";"3" **\crv{(0,-8) & (8,-8)}?(1)*\dir{<};
    "3"; (8,10) **\dir{-};
    (12,-9)*{\lambda+\alpha_i};
    (-6,9)*{\lambda};
    \endxy
    \; =
    \;
\xy   0;/r.17pc/:
    (-8,0)*{}="1";
    (0,0)*{}="2";
    (8,0)*{}="3";
    (0,-10);(0,10)**\dir{-} ?(.5)*\dir{<};
   (9,8)*{\lambda+\alpha_i};
    (-6,8)*{\lambda};
    \endxy
\end{equation}

The dots and crossings are cyclic with respect to this biadjoint structure:
\begin{equation} \label{eq_cyclic_dot}
 \xy 0;/r.17pc/:
    (-8,5)*{}="1";
    (0,5)*{}="2";
    (0,-5)*{}="2'";
    (8,-5)*{}="3";
    (-8,-10);"1" **\dir{-};
    "2";"2'" **\dir{-} ?(.5)*\dir{<};
    "1";"2" **\crv{(-8,12) & (0,12)} ?(0)*\dir{<};
    "2'";"3" **\crv{(0,-12) & (8,-12)}?(1)*\dir{<};
    "3"; (8,10) **\dir{-};
    (17,-9)*{\lambda+\alpha_i};
    (-12,9)*{\lambda};
    (0,4)*{\bullet};
    (10,8)*{\scs };
    (-10,-8)*{\scs };
    \endxy
    \quad = \quad
      \xy 0;/r.17pc/:
 (0,10);(0,-10); **\dir{-} ?(.75)*\dir{<}+(2.3,0)*{\scriptstyle{}}
 ?(.1)*\dir{ }+(2,0)*{\scs };
 (0,0)*{\bullet};
 (-6,5)*{\lambda};
 (10,5)*{\lambda+\alpha_i};
 (-10,0)*{};(10,0)*{};(-2,-8)*{\scs };
 \endxy
    \quad = \quad
   \xy 0;/r.17pc/:
    (8,5)*{}="1";
    (0,5)*{}="2";
    (0,-5)*{}="2'";
    (-8,-5)*{}="3";
    (8,-10);"1" **\dir{-};
    "2";"2'" **\dir{-} ?(.5)*\dir{<};
    "1";"2" **\crv{(8,12) & (0,12)} ?(0)*\dir{<};
    "2'";"3" **\crv{(0,-12) & (-8,-12)}?(1)*\dir{<};
    "3"; (-8,10) **\dir{-};
    (17,9)*{\lambda+\alpha_i};
    (-12,-9)*{\lambda};
    (0,4)*{\bullet};
    (-10,8)*{\scs };
    (10,-8)*{\scs };
    \endxy
\end{equation}
\begin{equation} \label{eq_almost_cyclic}
   \xy 0;/r.17pc/:
  (0,0)*{\xybox{
    (-4,4)*{};(4,-4)*{} **\crv{(-4,1) & (4,-1)}?(1)*\dir{>} ;
    (4,4)*{};(-4,-4)*{} **\crv{(4,1) & (-4,-1)}?(1)*\dir{>};
    (-6.5,-3)*{\scs i};
     (6.5,-3)*{\scs j};
     (9,1)*{\scs  \lambda};
     (-10,0)*{};(10,0)*{};
     }};
  \endxy \quad = \quad
  \xy 0;/r.17pc/:
  (0,0)*{\xybox{
    (4,-4)*{};(-4,4)*{} **\crv{(4,-1) & (-4,1)}?(1)*\dir{>};
    (-4,-4)*{};(4,4)*{} **\crv{(-4,-1) & (4,1)};
     (-4,4)*{};(18,4)*{} **\crv{(-4,16) & (18,16)} ?(1)*\dir{>};
     (4,-4)*{};(-18,-4)*{} **\crv{(4,-16) & (-18,-16)} ?(1)*\dir{<}?(0)*\dir{<};
     (-18,-4);(-18,12) **\dir{-};(-12,-4);(-12,12) **\dir{-};
     (18,4);(18,-12) **\dir{-};(12,4);(12,-12) **\dir{-};
     (22,1)*{ \lambda};
     (-10,0)*{};(10,0)*{};
     (-4,-4)*{};(-12,-4)*{} **\crv{(-4,-10) & (-12,-10)}?(1)*\dir{<}?(0)*\dir{<};
      (4,4)*{};(12,4)*{} **\crv{(4,10) & (12,10)}?(1)*\dir{>}?(0)*\dir{>};
      (-20,11)*{\scs j};(-10,11)*{\scs i};
      (20,-11)*{\scs j};(10,-11)*{\scs i};
     }};
  \endxy
\quad =  \quad
\xy 0;/r.17pc/:
  (0,0)*{\xybox{
    (-4,-4)*{};(4,4)*{} **\crv{(-4,-1) & (4,1)}?(1)*\dir{>};
    (4,-4)*{};(-4,4)*{} **\crv{(4,-1) & (-4,1)};
     (4,4)*{};(-18,4)*{} **\crv{(4,16) & (-18,16)} ?(1)*\dir{>};
     (-4,-4)*{};(18,-4)*{} **\crv{(-4,-16) & (18,-16)} ?(1)*\dir{<}?(0)*\dir{<};
     (18,-4);(18,12) **\dir{-};(12,-4);(12,12) **\dir{-};
     (-18,4);(-18,-12) **\dir{-};(-12,4);(-12,-12) **\dir{-};
     (22,1)*{ \lambda};
     (-10,0)*{};(10,0)*{};
      (4,-4)*{};(12,-4)*{} **\crv{(4,-10) & (12,-10)}?(1)*\dir{<}?(0)*\dir{<};
      (-4,4)*{};(-12,4)*{} **\crv{(-4,10) & (-12,10)}?(1)*\dir{>}?(0)*\dir{>};
      (20,11)*{\scs i};(10,11)*{\scs j};
      (-20,-11)*{\scs i};(-10,-11)*{\scs j};
     }};
  \endxy
\end{equation}

The relations below imply the usual
$\mathfrak{sl}_2$ commutator relations
between $E$ and $F$ on the level of the Grothendieck group.
\begin{equation} \label{eq_EF}
 \vcenter{\xy 0;/r.17pc/:
  (-8,0)*{};
  (8,0)*{};
  (-4,10)*{}="t1";
  (4,10)*{}="t2";
  (-4,-10)*{}="b1";
  (4,-10)*{}="b2";(-6,-8)*{\scs i};(6,-8)*{\scs i};
  "t1";"b1" **\dir{-} ?(.5)*\dir{<};
  "t2";"b2" **\dir{-} ?(.5)*\dir{>};
  (10,2)*{\l};
  \endxy}
\;\; = \;\; -\;\;
 \vcenter{   \xy 0;/r.17pc/:
    (-4,-4)*{};(4,4)*{} **\crv{(-4,-1) & (4,1)}?(1)*\dir{>};
    (4,-4)*{};(-4,4)*{} **\crv{(4,-1) & (-4,1)}?(1)*\dir{<};?(0)*\dir{<};
    (-4,4)*{};(4,12)*{} **\crv{(-4,7) & (4,9)};
    (4,4)*{};(-4,12)*{} **\crv{(4,7) & (-4,9)}?(1)*\dir{>};
  (8,8)*{\l};
     (-6,-3)*{\scs i};
     (6.5,-3)*{\scs i};
 \endxy}
  \;\; + \;\;
   \sum_{ \xy  (0,3)*{\scs f_1+f_2+f_3}; (0,0)*{\scs =\la i,\lambda\ra-1};\endxy}
    \vcenter{\xy 0;/r.17pc/:
    (-12,10)*{\l};
    (-8,0)*{};
  (8,0)*{};
  (-4,-15)*{}="b1";
  (4,-15)*{}="b2";
  "b2";"b1" **\crv{(5,-8) & (-5,-8)}; ?(.05)*\dir{<} ?(.93)*\dir{<}
  ?(.8)*\dir{}+(0,-.1)*{\bullet}+(-3,2)*{\scs f_3};
  (-4,15)*{}="t1";
  (4,15)*{}="t2";
  "t2";"t1" **\crv{(5,8) & (-5,8)}; ?(.15)*\dir{>} ?(.95)*\dir{>}
  ?(.4)*\dir{}+(0,-.2)*{\bullet}+(3,-2)*{\scs \; f_1};
  (0,0)*{\iccbub{\scs \quad -\la i,\lambda\ra-1+f_2}{i}};
  (7,-13)*{\scs i};
  (-7,13)*{\scs i};
  \endxy}
\end{equation}
\begin{equation}\label{eq_FE}
 \vcenter{\xy 0;/r.17pc/:
  (-8,0)*{};(-6,-8)*{\scs i};(6,-8)*{\scs i};
  (8,0)*{};
  (-4,10)*{}="t1";
  (4,10)*{}="t2";
  (-4,-10)*{}="b1";
  (4,-10)*{}="b2";
  "t1";"b1" **\dir{-} ?(.5)*\dir{>};
  "t2";"b2" **\dir{-} ?(.5)*\dir{<};
  (10,2)*{\l};
  (-10,2)*{\l};
  \endxy}
\;\; = \;\;
  -\;\;\vcenter{\xy 0;/r.17pc/:
    (-4,-4)*{};(4,4)*{} **\crv{(-4,-1) & (4,1)}?(1)*\dir{<};?(0)*\dir{<};
    (4,-4)*{};(-4,4)*{} **\crv{(4,-1) & (-4,1)}?(1)*\dir{>};
    (-4,4)*{};(4,12)*{} **\crv{(-4,7) & (4,9)}?(1)*\dir{>};
    (4,4)*{};(-4,12)*{} **\crv{(4,7) & (-4,9)};
  (8,8)*{\l};(-6.5,-3)*{\scs i};  (6,-3)*{\scs i};
 \endxy}
  \;\; + \;\;
    \sum_{ \xy  (0,3)*{\scs g_1+g_2+g_3}; (0,0)*{\scs =-\la i,\lambda\ra-1};\endxy}
    \vcenter{\xy 0;/r.17pc/:
    (-8,0)*{};
  (8,0)*{};
  (-4,-15)*{}="b1";
  (4,-15)*{}="b2";
  "b2";"b1" **\crv{(5,-8) & (-5,-8)}; ?(.1)*\dir{>} ?(.95)*\dir{>}
  ?(.8)*\dir{}+(0,-.1)*{\bullet}+(-3,2)*{\scs g_3};
  (-4,15)*{}="t1";
  (4,15)*{}="t2";
  "t2";"t1" **\crv{(5,8) & (-5,8)}; ?(.15)*\dir{<} ?(.9)*\dir{<}
  ?(.4)*\dir{}+(0,-.2)*{\bullet}+(3,-2)*{\scs g_1};
  (0,0)*{\icbub{\scs \quad\; \la i,\lambda\ra-1 + g_2}{i}};
    (7,-13)*{\scs i};
  (-7,13)*{\scs i};
  (-10,10)*{\l};
  \endxy}
\end{equation}

For $i \neq j$, we have
\begin{equation}
 \vcenter{\xy 0;/r.17pc/:
    (-4,-4)*{};(4,4)*{} **\crv{(-4,-1) & (4,1)}?(1)*\dir{};
    (4,-4)*{};(-4,4)*{} **\crv{(4,-1) & (-4,1)}?(1)*\dir{};
    (-4,4)*{};(4,12)*{} **\crv{(-4,7) & (4,9)}?(1)*\dir{};
    (4,4)*{};(-4,12)*{} **\crv{(4,7) & (-4,9)}?(1)*\dir{};
    (8,8)*{\lambda};
    (4,12); (4,13) **\dir{-}?(1)*\dir{>};
    (-4,12); (-4,13) **\dir{-}?(1)*\dir{>};
  (-5.5,-3)*{\scs i};
     (5.5,-3)*{\scs j};
 \endxy}
 \qquad = \qquad
 \left\{
 \begin{array}{ccc}
 0 & & \text{if $(\alpha_i,\alpha_j)=2$, } \\ \\
     t_{ij}\;\xy 0;/r.17pc/:
  (3,9);(3,-9) **\dir{-}?(0)*\dir{<}+(2.3,0)*{};
  (-3,9);(-3,-9) **\dir{-}?(0)*\dir{<}+(2.3,0)*{};
  (-5,-6)*{\scs i};     (5.1,-6)*{\scs j};
 \endxy &  &  \text{if $(\alpha_i, \alpha_j)=0$,}\\ \\
 t_{ij} \vcenter{\xy 0;/r.17pc/:
  (3,9);(3,-9) **\dir{-}?(0)*\dir{<}+(2.3,0)*{};
  (-3,9);(-3,-9) **\dir{-}?(0)*\dir{<}+(2.3,0)*{};
  (-3,4)*{\bullet};(-6.5,5)*{};
  (-5,-6)*{\scs i};     (5.1,-6)*{\scs j};
 \endxy} \;\; + \;\; t_{ji}
  \vcenter{\xy 0;/r.17pc/:
  (3,9);(3,-9) **\dir{-}?(0)*\dir{<}+(2.3,0)*{};
  (-3,9);(-3,-9) **\dir{-}?(0)*\dir{<}+(2.3,0)*{};
  (3,4)*{\bullet};(7,5)*{};
  (-5,-6)*{\scs i};     (5.1,-6)*{\scs j};
 \endxy}
   &  & \text{if $(\alpha_i, \alpha_j) =-1$,}
 \end{array}
 \right. \label{eq_r2_ij-gen}
\end{equation}

 For $i \neq j$ the dot sliding relations look as follows:
\begin{equation} \label{eq_dot_slide_ij-gen}
\xy 0;/r.18pc/:
  (0,0)*{\xybox{
    (-4,-4)*{};(4,6)*{} **\crv{(-4,-1) & (4,1)}?(1)*\dir{>}?(.75)*{\bullet};
    (4,-4)*{};(-4,6)*{} **\crv{(4,-1) & (-4,1)}?(1)*\dir{>};
    (-5,-3)*{\scs i};
     (5.1,-3)*{\scs j};
     (-10,0)*{};(10,0)*{};
     }};
  \endxy
 \;\; =
\xy 0;/r.18pc/:
  (0,0)*{\xybox{
    (-4,-4)*{};(4,6)*{} **\crv{(-4,-1) & (4,1)}?(1)*\dir{>}?(.25)*{\bullet};
    (4,-4)*{};(-4,6)*{} **\crv{(4,-1) & (-4,1)}?(1)*\dir{>};
    (-5,-3)*{\scs i};
     (5.1,-3)*{\scs j};
     (-10,0)*{};(10,0)*{};
     }};
  \endxy
\qquad  \xy 0;/r.18pc/:
  (0,0)*{\xybox{
    (-4,-4)*{};(4,6)*{} **\crv{(-4,-1) & (4,1)}?(1)*\dir{>};
    (4,-4)*{};(-4,6)*{} **\crv{(4,-1) & (-4,1)}?(1)*\dir{>}?(.75)*{\bullet};
    (-5,-3)*{\scs i};
     (5.1,-3)*{\scs j};
     (-10,0)*{};(10,0)*{};
     }};
  \endxy
\;\;  =
  \xy 0;/r.18pc/:
  (0,0)*{\xybox{
    (-4,-4)*{};(4,6)*{} **\crv{(-4,-1) & (4,1)}?(1)*\dir{>} ;
    (4,-4)*{};(-4,6)*{} **\crv{(4,-1) & (-4,1)}?(1)*\dir{>}?(.25)*{\bullet};
    (-5,-3)*{\scs i};
     (5.1,-3)*{\scs j};
     (-10,0)*{};(12,0)*{};
     }};
  \endxy
\end{equation}
hold.

 When $i \ne j$ one has the mixed relations  relating $\cal{E}_i \cal{F}_j$ and $\cal{F}_j \cal{E}_i$:
\begin{equation} \label{mixed_rel}
 \vcenter{   \xy 0;/r.18pc/:
    (-4,-4)*{};(4,4)*{} **\crv{(-4,-1) & (4,1)}?(1)*\dir{>};
    (4,-4)*{};(-4,4)*{} **\crv{(4,-1) & (-4,1)}?(1)*\dir{<};?(0)*\dir{<};
    (-4,4)*{};(4,12)*{} **\crv{(-4,7) & (4,9)};
    (4,4)*{};(-4,12)*{} **\crv{(4,7) & (-4,9)}?(1)*\dir{>};
  (8,8)*{\lambda};(-6,-3)*{\scs i};
     (6,-3)*{\scs j};
 \endxy}
 \;\; = \;\; \;\;
\xy 0;/r.18pc/:
  (3,9);(3,-9) **\dir{-}?(.55)*\dir{>}+(2.3,0)*{};
  (-3,9);(-3,-9) **\dir{-}?(.5)*\dir{<}+(2.3,0)*{};
  (8,2)*{\lambda};(-5,-6)*{\scs i};     (5.1,-6)*{\scs j};
 \endxy
\qquad \quad
    \vcenter{\xy 0;/r.18pc/:
    (-4,-4)*{};(4,4)*{} **\crv{(-4,-1) & (4,1)}?(1)*\dir{<};?(0)*\dir{<};
    (4,-4)*{};(-4,4)*{} **\crv{(4,-1) & (-4,1)}?(1)*\dir{>};
    (-4,4)*{};(4,12)*{} **\crv{(-4,7) & (4,9)}?(1)*\dir{>};
    (4,4)*{};(-4,12)*{} **\crv{(4,7) & (-4,9)};
  (8,8)*{\lambda};(-6,-3)*{\scs i};
     (6,-3)*{\scs j};
 \endxy}
 \;\;=\;\; \;\;
\xy 0;/r.18pc/:
  (3,9);(3,-9) **\dir{-}?(.5)*\dir{<}+(2.3,0)*{};
  (-3,9);(-3,-9) **\dir{-}?(.55)*\dir{>}+(2.3,0)*{};
  (8,2)*{\lambda};(-5,-6)*{\scs i};     (5.1,-6)*{\scs j};
 \endxy
\end{equation}



The curl can be expressed as a linear combination of bubbles:

\begin{equation} \label{eq_curl}
  \xy 0;/r.17pc/:
  (14,8)*{\l};
  (-3,-10)*{};(3,5)*{} **\crv{(-3,-2) & (2,1)}?(1)*\dir{>};?(.15)*\dir{>};
    (3,-5)*{};(-3,10)*{} **\crv{(2,-1) & (-3,2)}?(.85)*\dir{>} ?(.1)*\dir{>};
  (3,5)*{}="t1";  (9,5)*{}="t2";
  (3,-5)*{}="t1'";  (9,-5)*{}="t2'";
   "t1";"t2" **\crv{(4,8) & (9, 8)};
   "t1'";"t2'" **\crv{(4,-8) & (9, -8)};
   "t2'";"t2" **\crv{(10,0)} ;
   (-6,-8)*{\scs i};
 \endxy \;\; = \;\; -
   \sum_{ \xy  (0,3)*{\scs f_1+f_2+f_3}; (0,0)*{\scs =-\la i,\l\ra};\endxy}
 \xy 0;/r.17pc/:
  (19,4)*{\l};
  (0,0)*{\bbe{}};(-2,-8)*{\scs };
  (-2,-8)*{\scs i};
  (12,-2)*{\icbub{\la i,\l\ra-1+f_2}{i}};
  (0,6)*{\bullet}+(3,-1)*{\scs f_1};
 \endxy
\qquad
  \xy 0;/r.17pc/:
  (-14,8)*{\l};
  (3,-10)*{};(-3,5)*{} **\crv{(3,-2) & (-2,1)}?(1)*\dir{>};?(.15)*\dir{>};
    (-3,-5)*{};(3,10)*{} **\crv{(-2,-1) & (3,2)}?(.85)*\dir{>} ?(.1)*\dir{>};
  (-3,5)*{}="t1";  (-9,5)*{}="t2";
  (-3,-5)*{}="t1'";  (-9,-5)*{}="t2'";
   "t1";"t2" **\crv{(-4,8) & (-9, 8)};
   "t1'";"t2'" **\crv{(-4,-8) & (-9, -8)};
   "t2'";"t2" **\crv{(-10,0)} ;
   (6,-8)*{\scs i};
 \endxy \;\; = \;\;
 \sum_{ \xy  (0,3)*{\scs g_1+g_2+g_3}; (0,0)*{\scs =\la i,\l\ra};\endxy}
  \xy 0;/r.17pc/:
  (5,-8)*{\scs i};
  (-12,8)*{\l};
  (3,0)*{\bbe{}};(2,-8)*{\scs};
  (-12,-2)*{\iccbub{-\la i,\l\ra-1+g_2}{i}};
  (0,6)*{\bullet}+(3,-1)*{\scs g_1};
 \endxy
\end{equation}

%
\subsection{Bubbles and symmetric functions}
It is well known that for a partition $\lambda = (\lambda_1, \dots, \lambda_n)$ with $\lambda_1 \geq \lambda_2 \geq \dots \geq \lambda_n$, products of elementary symmetric functions $e_{\lambda} = e_{\lambda_1} \dots e_{\lambda_n}$ form a $\Z$-basis for
symmetric functions
${\rm Sym}$, see for example \cite{McD}.  Likewise, products of complete symmetric functions also provide a $\Z$-basis for $\sym$.  This mirrors the fact that any closed diagram in the graphical calculus for $\Ucat_Q(\mf{sl}_2)$ can be reduced to a product on non-nested bubbles of a given orientation.

In the calculus of the 2-category $\Ucat_Q(\mf{g})$,  we have the isomorphism
$$ \psi_{\lambda} \maps \;\; \prod_{i\in I}{\rm Sym}  \;\; \longrightarrow\;\; Z(\l)= \Ucat_Q(\mf{g})(\onel,\onel) $$
 since any closed diagram can still be reduced to products of
 non-nested closed bubbles labelled by $i \in I$, and
 \cite[Th. A]{Webunfurl} shows the source and image have the same
 graded dimension.

In what follows, it will be interesting to consider which products of closed diagrams correspond to the $\Q$-basis of $\sym$ given by the power sum $p_r$ symmetric functions (see e.g. p.16 in \cite{McD}).
Using a formula that expresses power sum symmetric functions in terms of products of complete and elementary symmetric functions, we can denote by $p_{i,r}(\lambda)$ for $r>0$, the image of the power sum symmetric polynomial on $i$-labelled strands in $Z(\l)$:
\begin{equation} \label{eq_defpil}
 p_{i,r}(\lambda) := \sum_{a+b=r} (a+1)
 \xy
   (0,-2)*{\icbub{\spadesuit+a}{i}};
   (12,-2)*{\iccbub{\spadesuit+b}{i}};
  (8,8)*{\lambda};
 \endxy
 =  - \sum_{a+b=r} (b+1)
 \xy
   (0,-2)*{\icbub{\spadesuit+a}{i}};
   (12,-2)*{\iccbub{\spadesuit+b}{i}};
  (8,8)*{\lambda};
 \endxy =  -\sum_{a+b=r} a
 \xy
   (0,-2)*{\icbub{\spadesuit+b}{i}};
   (12,-2)*{\iccbub{\spadesuit+a}{i}};
  (8,8)*{\lambda};
 \endxy
\end{equation}
For later convenience we set $p_{i,0}(\l) = \la i,\lambda\ra$.


The bubble sliding equations imply the following power sum slide rule
\begin{eqnarray}
    \label{eq_powerslide2}
   \xy 0;/r.18pc/:
  (10,6)*{\lambda};
  (0,0)*{\bbe{}};
  (2,-7)*{\scs j};
  (-14,-2)*{p_{i,r}(\lambda+\alpha_j)};
  (0,6)*{ }+(7,-1)*{\scs  };
 \endxy
 & = &
 \begin{cases}
\xy 0;/r.18pc/:
  (14,6)*{\lambda};
  (0,0)*{\bbe{}};
  (0,-12)*{\scs j};
  (10,-2)*{p_{i,r}(\lambda)};
  (0,6)*{ }+(7,-1)*{\scs  };
 \endxy
 \;\; + \;\; 2\;\;\xy 0;/r.18pc/:
  (10,6)*{\lambda};
  (0,0)*{\bbe{}};
  (2,-7)*{\scs j};
  (0,2)*{\bullet}+(3,1)*{r};
  (0,6)*{ }+(7,-1)*{\scs  };
 \endxy & \mbox{if } i = j, \\ & \\
\xy 0;/r.18pc/:
  (14,6)*{\lambda};
  (0,0)*{\bbe{}};
  (0,-12)*{\scs j};
  (10,-2)*{p_{i,r}(\lambda)};
  (0,6)*{ }+(7,-1)*{\scs  };
 \endxy
 \;\; -  \;\; (-v_{ij})^r\;\;\xy 0;/r.18pc/:
  (10,6)*{\lambda};
  (0,0)*{\bbe{}};
  (2,-7)*{\scs j};
  (0,2)*{\bullet}+(3,1)*{r};
  (0,6)*{ }+(7,-1)*{\scs  };
 \endxy
 & \mbox{if } a_{ij} = -1
 \end{cases}
\end{eqnarray}

%
\subsection{The $2$-category $\Ucat^*_Q(\mf{g})$}
\label{sec:Ustar}
%

The $2$-category $\Ucat^{\ast}:=\Ucat^*_Q(\mf{g})$ is defined as follows.  The objects and
$1$-morphisms are the same as those of $\U_Q(\mf{g})$.  Given a pair
of $1$-morphisms $f,g\col n\to m$, the abelian group $\U^*(n,m)(f,g)$
is defined by
\begin{gather*}
  \U^*(n,m)(f,g):= \bigoplus_{t\in\Z}\U(n,m)(f,g\la t\ra).
\end{gather*}
The category $\U^*(n,m)$ is additive and enriched over $\Z$-graded
abelian groups.  Alternatively, the linear category $\U^*(n,m)$ is
obtained from $\U(n,m)$ by adding a family of natural isomorphisms
$f\rightarrow f\la1\ra$ for each object $f$ of the category $\U(n,m)$.

In $\U^*(n,m)$ an object $f$ and its translation $f\la t\ra$
are isomorphic via the $2$-isomorphism
\begin{gather*}
  1_{f}\in \U(n,m)(f,f\la 0\ra)=\U(n,m)(f,(f\la t\ra) \la -t \ra)\subset\U^*(n,m)(f,f\la t\ra).
\end{gather*}
The inverse of the isomorphism $1_f \maps f \to f \la t \ra$ is given by
\begin{gather*}
  1_{f}\la t \ra \in \U(n,m)(f\la t \ra,f\la t\ra)=\U(n,m)(f\la t\ra,(f\la 0\ra) \la t \ra)\subset\U^*(n,m)(f\la t\ra ,f).
\end{gather*}
These isomorphisms $f \cong f\la t\ra$ make the Grothendieck group
$K_0(\U^{\ast})$ show that $[f]_{\cong}=[f\la t \ra ]_{\cong}$ in
$\U^{\ast}$.  Thus, the $\Z[q,q^{-1}]$-module structure
satisfies the equation $q=1$, and thus factors through the quotient
$\Z$, so it is more natural to only consider it as an abelian group.

The horizontal composition in $\U$ induces horizontal composition in
$\U^*$.  It follows that the $\U^*(n,m)$, $n,m\in\Z$, form an additive
$2$-category.

The Karoubi envelope $\Kar(\U^*)$ will be denoted by $\dot{\Ucat}^*$, which
is equivalent as an additive $2$-category to the additive $2$-category
obtained from $\dot{\Ucat}$ by defining
\begin{gather*}
  \dot{\Ucat}^*(n,m)(f,g) = \bigoplus_{t\in\Z}\dot{\Ucat}(n,m)(f,g\la t\ra).
\end{gather*}

%
\section{A homomorphism from the current algebra}
\label{current-algebra}
%

In type $A$ a homomorphism from the current algebra to the trace of the 2-category $\U^*$
was constructed in \cite{BHLZ} and \cite{BGHL}. For $\fsl_2$ is it proven to be
an isomorphism for the integral version of $\U^*$ in  \cite{BHLZ}.

 The image
$\mathsf{E}_{i,r}^{(a)}$ of the divided power of the current algebra generator is
given by  $[y_1^r\cdots y_a^re_a]$, which is the
$r$th power of a dot on each of $a$ consecutive strands, multiplied by
a primitive idempotent in the nilHecke algebra (see Proposition 9.7 and Corollary 9.8 in \cite{BHLZ}).
Since any two primitive idempotents are
equal in the trace, and the identity is the sum of $a!$ many such
idempotents, this indeed satisfies $\mathsf{E}_{i,r}^{a}=(a!) \mathsf{E}_{i,r}^{(a)}$.

The relevant parameters that govern the behavior of the 2-category
$\Ucat_Q(\mf{g})$ are the products $v_{ij}=t_{ij}^{-1}t_{ji}$ taken
over all pairs $i,j\in I$.
We will
assume throughout in this paper that we have chosen these scalars so
that $v_{ij}=(-1)^{a_{ij}}$.  For simply laced
Cartan datum,
this
means that $v_{ij}=-1$ whenever $i$ and $j$
are connected by an edge, and $v_{ij}=1$ if they are not.  The
most natural method for making such a choice is to orient the Dynkin
diagram, and set $t_{ij}=-t_{ji}=1$ if there is an oriented edge $j\to
i$.
These assumptions
will prove necessary for the result below.
\begin{prop}\label{prop-cur}
Let \(\mathsf{E}_{i,r}1_{\lambda}, \mathsf{F}_{j,s}1_{\lambda}, \mathsf{H}_{i,r}1_{\lambda}\) denote
the elements of \(\Tr(\U^\ast_Q(\mf{g}))\):
\[\mathsf{E}_{i,r}1_{\lambda}:=  \left[
\xy 0;/r.18pc/:
  (10,6)*{\lambda};
  (0,0)*{\bbe{}};
  (0,2)*{\bullet}+(3,1)*{r};
  (-8,6)*{ }; (12,-7)*{}; (2,-7)*{\scs i};
 \endxy
\right],  \quad \quad
\mathsf{F}_{j,s}1_{\lambda}:=  \left[
\xy 0;/r.18pc/:
  (10,6)*{\lambda};
  (0,0)*{\bbf{}};
  (2,-7)*{\scs j};
  (0,4)*{\bullet}+(3,1)*{s};
   (-8,6)*{ }; (12,-7)*{};
 \endxy
\right],\quad \quad
  \mathsf{H}_{i,r}1_{\lambda}:= \left[
 p_{i,r}(\lambda)\; \Id_{\onel}
\right],\]
where $p_{i,r}(\l)$ was defined in equation~\eqref{eq_defpil}.
There is a linear functor
\begin{equation} \label{eq_sln-homo}
 \rho \maps \dcurkg\longrightarrow \Tr(\U^\ast_Q(\mf{g})),
\end{equation}
given by
\begin{equation} \label{homomorp}
(x^{+}_{i, r})^{(a)} 1_{\lambda} \mapsto
\mathsf{E}_{i,r}^{(a)} 1_{\lambda}
 ,\quad \quad
(x^{-}_{j, s}) ^{(a)} 1_{\lambda}\mapsto
\mathsf{F}_{j,s}^{(a)} 1_{\lambda}
,\quad \quad
\xi_{i, r} 1_{\lambda}\mapsto
\mathsf{H}_{i,r}1_{\lambda}.
\end{equation}
\end{prop}

We will denote by $\rho^\pm$ and $\rho^0$ the restrictions of $\rho$ to the subcategories
 $\dot{\bfU}^\pm(\mf{g}[t])$ and $\dot{\bfU}^0(\mf{g}[t])$,
 respectively.


\begin{proof}
To prove this proposition we verify the current algebra relations  using the relations in the 2-category $\U^\ast_Q(\mf{g})$. We only need
to consider the case \(i\neq j\), since the relations in $\U^\ast_Q(\fsl_2)$ have been proven in \cite{BHLZ}.
\eqref{C1} is clear, since bubbles commute with each other.  Axiom \eqref{C2} follows immediately from the definition of $p_{i,0}\onel = \la i, \lambda\ra\onel$.

Consider the equality \eqref{C3}.
The case $a_{ij}=0$ follows easily from the bubble slide relation.
Suppose \(a_{ij}=-1\).
Using the power sum slide identity
\eqref{eq_powerslide2},
we get
\[ \mathsf{H}_{i,r}\mathsf{E}_{j,s} 1_\lambda\;\; = \;\;
   \mathsf{E}_{j,s} \mathsf{H}_{i,r} 1_\lambda - (-v_{ij})^r\;\mathsf{E}_{j,r+s}  1_\lambda\;\; = \;\;
   \mathsf{E}_{j,s} \mathsf{H}_{i,r} 1_\lambda - \;\mathsf{E}_{j,r+s}  1_\lambda.\]

The relation\[[ \mathsf{H}_{i,r}, \mathsf{F}_{j,s}] 1_\lambda = (-v_{ij})^r \mathsf{F}_{j,r+s}1_\lambda =\mathsf{F}_{j,r+s}1_\lambda .\]
can be proven in a similar way.  This verifies \eqref{C3}.

The relation \eqref{C4} follows from the relations
\eqref{eq_r2_ij-gen} and \eqref{eq_dot_slide_ij-gen}, which imply
\[ [ \mathsf{E}_{i,r+1},\mathsf{E}_{j,s} ]1_\lambda =-v_{ij}[\mathsf{E}_{i,r},\mathsf{E}_{j,s+1} ]1_\lambda =[\mathsf{E}_{i,r},\mathsf{E}_{j,s+1} ]1_\lambda.\]
  If we reverse the arrows in the preceding equation, they still hold:
   \[ [ \mathsf{F}_{i,r+1},\mathsf{F}_{j,s} ]1_\lambda=-v_{ij}[\mathsf{F}_{i,r},\mathsf{F}_{j,s+1} ]1_\lambda=[\mathsf{F}_{i,r},\mathsf{F}_{j,s+1} ]1_\lambda .\]
The choice of signs  in \eqref{homomorp} corresponds to
  \[ [x^\pm_{i, r+1}, x^\pm_{j,s}]=[x^\pm_{i,r}, x^\pm_{j,s+1}].\]

The relation \eqref{C5} for the case \(i\neq j\) follows from the relations
\eqref{eq_dot_slide_ij-gen} and \eqref{mixed_rel} which can be used to show
\[
\mathsf{E}_{i,r}\mathsf{F}_{j,s} 1_\lambda
 \;\; = \;\;
\mathsf{F}_{j,s} \mathsf{E}_{i,r}1_\lambda.
\]

The relation \eqref{C6a} follows since
\[
 x_{i,r}^{+} x_{j,s}^{+}
 = \mathsf{E}_{i,r}\mathsf{E}_{j,s}
 =t_{ij}^{-1}t_{ji}\mathsf{E}_{j,s}\mathsf{E}_{i,r}
 = x_{j,s}^{+} x_{i,r}^{+} ,
\]
since $t_{ij}=t_{ji}$ when $a_{ij}=0$.   The case $[x_{i,r}^{-} ,x_{j,s}^{-}]$ is proven similarly.
To verify \eqref{C6b}, the color dependent rescalings play no role. By  a direct computation we get
\[
\mathsf{E}_{i,r_1}\mathsf{E}_{j,s}\mathsf{E}_{i,r_2}1_{\lambda} +\mathsf{E}_{i,r_2}\mathsf{E}_{j,s}\mathsf{E}_{i,r_1}1_{\lambda} = \mathsf{E}_{i,r_1}\mathsf{E}_{i,r_2}\mathsf{E}_{j,s} 1_{\lambda}+\mathsf{E}_{j,s}\mathsf{E}_{i,r_1}\mathsf{E}_{i,r_2} 1_{\lambda}.
\qedhere\]
\end{proof}

%
\subsection{Triangular decomposition}
%

Let $\Ucat_{\rm {tr}}^+=\Ucat^+_{{\rm {tr}},Q}$ denote the $\Bbbk$-linear subcategory of $\Tr(\Ucat^{\ast}_Q(\mf{g}))$ with objects indexed by the weight lattice
$\Ob(\Ucat^+_{\rm {tr}}) = X$, and with morphisms generated by composites of $\mathsf{E}_{i,a}$ for $i \in I$ and $a \geq 0$.  Similarly,
let $\Ucat^-_{\rm {tr}}=\Ucat^-_{\rm {tr}, Q}$ denote the $\Bbbk$-linear subcategory of $\Tr(\Ucat^{\ast}_Q(\mf{g}))$ with objects $\Ob(\Ucat^-_{\rm {tr}}) = X$ and morphisms generated by $\mathsf{F}_{j,b}$ for $j \in I$ and $b\geq 0$.   We define $\Ucat^0_{\rm {tr}}=\Ucat^0_{\rm {tr},Q}$ as the $\Bbbk$-linear subcategory of $\Tr(\Ucat^{\ast}_Q(\mf{g}))$ with objects $\Ob(\Ucat^0_{\rm {tr}}) = X$ and morphisms generated by bubbles.
So for any $\l\in X$, we have  $1_\l\Ucat^0_{\rm {tr}}1_\l\cong
{\rm{Sym}}^{\otimes |I|}$,
where $|I|$ is the cardinality of $I$.
\begin{prop}
  The map $\rho^0\maps\dcurzg\to\Ucat^0_{\rm {tr}}$ is an isomorphism.
\end{prop}

\begin{proof}
  We will show below that the $\Z$-version of the proposition holds.  The
  $\k$-version follows by tensoring with $\k$.

  For each $\l$, we have a surjective map of the $\Z$-forms
  \begin{samepage}
    \begin{align*}
      \eta^0 \colon {\bf{\dot
      U}}^0(\mf{sl}_2[t])^{\otimes|I|}&\to
                                        {\dot{\mathscr{C}}_\Z^0\mf{g}}:=\dcurzg\cap{\dot{\mathscr{C}}_\Z\mf{g}}\\
      1\otimes \cdots \otimes \xi_r\otimes \cdots \otimes 1&\mapsto \xi_{i,r}.
    \end{align*}
  \end{samepage}
By Lemma 8.2 in \cite{BHLZ}, in
  the $\fsl_2$ case, $\rho^0$ maps the idempotented version of
  Garland's $\Z$-form, ${\bf{\dot U}}^0(\mf{sl}_2[t])$
  isomorphically onto $\rm Sym$.  Therefore, in our case, $\rho^0
  \eta^0$
  maps $\dcurzg\cong{\bf{\dot U}}^0(\mf{sl}_2[t])^{\otimes|I|}$
  isomorphically onto ${\rm Sym}^{\otimes|I|}$, which is isomorphic to
  $1_\l\Ucat^0_{\rm {tr}}1_\l$ as we observed above.  Since $\rho^0
  \eta^0$ is an isomorphism and $\eta^0$ is surjective, we must have that $\rho^0$
  is an isomorphism.
\end{proof}



\begin{prop}\label{prop:triangular}
Let $f$ be a $2$-endomorphism in $\Ucat^{\ast}(\mf{g})$.
Then the class $[f]$ in $\Tr(\Ucat^{\ast}(\mf{g}))$ can be expressed as a sum
\[
 [f] =\sum [f^0] [f^+] [f^-]
\]
where $[f^\pm]$, and $[f^0]$  are in $\Ucat_{\rm {tr}}^{\pm}$, $\Ucat_{\rm {tr}}^0$,  respectively.
\end{prop}

\begin{proof}
Assume $f$ is a 2-endomorphism  of a sequence $\ii \in I^{|\ii|}$ of generating 1-morphisms $\E_i\onel$ and $\F_i\onel$
in $\U^\ast$ of length $|\ii|$. We'll prove this result by induction
on $|\ii|$ and the number of crossings in the diagram.  If $|\ii|=0$, then $[f]=[f^0].$

One way to think of taking the trace class $[f]$ is that we can slice
through this diagram at some horizontal level and then move the
portion from the top to the bottom; we refer to this as a ``cut and switch.''  Let us say the top of the diagram
is at $y=1$ and the bottom at $y=-1$.  Using isotopies and the spanning set of
\cite[\S 3.2.3]{KL3}, we can assume that all bubbles are at the far
left\footnote{In \cite{KL3}, they are at far right, but there is an
  auto-equivalence of the category switching these.}, and
the number of strands crossed by $y=0$ is minimal amongst horizontal
lines.  If this number is lower than $|\ii|$, then we can cut at
$y=0$ and switch the two halves, and reduce the size of $|\ii|$.  Thus,
we can assume that no strand in $\ii$ forms a cup or a cap.

Now, we can rewrite
our diagram as a sum of diagrams where at $y=0$, reading
from left to right, we encounter all upward oriented strands before we
encounter any downward ones.  If not, then there is a consecutive pair
of strands where the leftward one (of label $i$) is downward, and the
rightward one (of label $j$)
is upward.  Using \eqref{mixed_rel} if $i\neq j$, and  \eqref{eq_FE}
if $i=j$, we can rewrite this diagram as a sum of diagrams with fewer
out of order pairs.  Applying this inductively, we can reduce the
number of such pairs to 0.

Now, we cut and switch the resulting diagrams. All but
one term has  $|\ii'|<|\ii|$, so by induction these have the desired
form.  One remains with $|\ii'|=|\ii|$, with $\ii'=\ii^+\ii^-$, with
$\ii^\pm$ only containing elements of $\pm I$.   The spanning set of
\cite[\S 3.2.3]{KL3} shows that a diagram connecting these can be written as sum
of diagrams of the
form $f^0\hor f^+\hor f^-$, and diagrams with cups and caps.  As
before, we can cut and flip any diagram with a cup or cap, and use
the inductive hypothesis.

Thus, ultimately, we have that $[f]=\sum [f^0][ f^+][
f^-]$ plus diagrams with representatives given by 2-endomorphisms of
$\ii'$ with $|\ii'|<|\ii|$.  By induction, this completes the proof.
\end{proof}

%
\section{Surjectivity results}
\label{surjectivity}
%
\subsection{Surjectivity of $\rho$}
%

\begin{thm}\label{thm:minus-surjection}
  The linear functor $\rho^-\maps \dot{\mathscr{C}}^-\!\mf{g}\to \Ucat^-_{\rm {tr}}$ is an isomorphism on objects and full on morphisms (surjective as an algebra homomorphism).
\end{thm}

\newcommand{\Bi}{\mathbf{i}}
\newcommand{\Ba}{\mathbf{a}}
\newcommand{\Bb}{\mathbf{b}}
Choose any infinite sequence $\Bi=(i_1,i_2,\dots)\in I^{\Z_{>0}}$ such
that every element of $I$ appears an infinite number of times.  For
any infinite word in the integers, $\Ba=(a_1,a_2,\dots)$ with almost
all $a_i=0$, we let
$\Bi_\Ba$ be the concatenated word $i_1^{a_1}i_2^{a_2}\cdots$, and let
$\cal{F}_\Ba$ be the associated 1-morphism in $  \Ucat^-$.

Use induction over the lexicographic order on sequences to show that
 $\Ucat^-_{\rm {tr}}$ is spanned by dots.

 \begin{proof}
   We will give an inductive proof of the following statement:
   \begin{itemize}
   \item [$(*_\Ba)$] The image of $\End(\sF_{\Ba})$ in $\Ucat^-_{\rm {tr}}$  lies in the
     sum of
     images of the polynomial endomorphisms of $\End(\sF_\Bb)$ for
     $\Bb\geq \Ba$.
   \end{itemize}
Since every 2-morphism in $\Ucat^-$ factors through a finite sum of
$\sF_\Ba$'s, establishing this for every $\Ba$ will complete the proof.

First, assume that only one entry of $\Ba$ is non-zero.  In this case,
$\End(\sF_\Ba)$ is a nil-Hecke algebra, and thus has trace generated by
its polynomial subalgebra, as proven in \cite[9.8]{BHLZ}.

Now assume that $\Ba$ is arbitrary.  Every endomorphism of $\sF_\Ba$ can
be written as a sum of diagrams, so we may as well consider the case
of a single diagram $D$.  If the diagram has no crossings, it is
polynomial, and we are done.  Now having fixed $\Ba$, we induct on the
number of crossings.  Modulo elements with a lower number of crossings
than $D$, we can isotope the strands of $D$ though crossings.  In
particular, for some $k$, we can assume that the leftmost $k-1$ strands have no
crossings, and that the strands which $k$th from the left at the top
and bottom cross.  Let us call these
strands $U$ and $V$. We can further assume that all crossings occur to
the right of $U$ and of $V$ (or on the strands themselves), in the
region marked $D'$ in the diagram below.  We will now also induct (upward) on $k$.

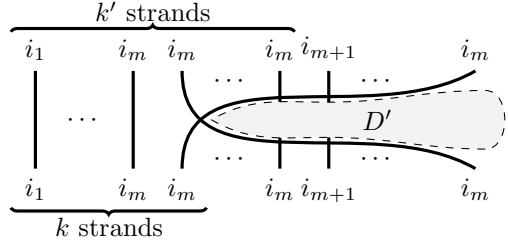
\[\tikz[very thick,scale=1.3]{
\draw (0,-.5) -- node[ below,at start]{$i_1$} node[ above,at end]{$i_1$}
(0,.5);
\node at (.5,0){$\cdots$};
\draw (1,-.5) -- node[ below,at start]{$i_{m}$} node[ above,at end]{$i_{m}$}
(1,.5);
\draw (1.5,-.5) to[out=90,in=-150] node[ below,at start]{$i_{m}$} node[ above,at end]{$i_{m}$}
(4.5,.5);
\draw (4.5,-.5) to[out=150,in=-90] node[ below,at start]{$i_{m}$} node[ above,at end]{$i_{m}$}
(1.5,.5);
\node at (2,-.4){$\cdots$};
\node at (2,.4){$\cdots$};
\draw (2.5,-.5) -- node[ below,at start]{$i_{m}$} node[ above,at end]{$i_{m}$}
(2.5,.5);
\draw (3,-.5) -- node[ below,at start]{$i_{m+1}$} node[ above,at end]{$i_{m+1}$}
(3,.5);
\node at (3.5,-.4){$\cdots$};
\node at (3.5,.4){$\cdots$};
\fill[gray!10!white,draw=black,thin, dashed] (4.6,-.3) to[out=180,in=0] (3,-.18) to[out=180,in=-30] (1.8,0)
to[out=30,in=180] (3,.18) to[out=0,in=180] (4.6,.3) to [out=0,in=90]
(4.8,0) to[out=-90,in=0] (4.6,-.3);
\node at (3.5,0){$D'$};
\draw[decorate,decoration=brace,-] (1.75,-.9) --
    node[below,midway]{$k$ strands} (-.25,-.9);
\draw[decorate,decoration=brace,-] (-.25,.9) --
    node[above,midway]{$k'$ strands} (2.65,.9);
}\]

 Let $m$ be the
smallest integer such that $k'=a_1+\cdots +a_m \geq k$.
Note that by definition, the strands between the $k$th and $k'$th from
the left are the same color by definition.  Thus, if the strands $U$
and $V$ have both ends left of the $k'$th strand, we can get rid of
their crossing, and increase $k$.  Thus, we can assume that $U$ and
$V$ do have one end at the $k$th terminal from left and their other at
a terminal further right than the $k'$th.

We wish to show that $D$ factors through $\sF_\Bb$, where $\Bb=(a_1,
a_2,\dots, a_m+1,\dots)$, which is thus higher in lexicographic order.
Consider that after crossing $V$, the strand $U$ crosses the $k+1$st
strand from the left, the $k+2$nd, etc. until it reaches the $k'$th.
Below $U$, these other strands don't cross; thus, they all have the
same label at $U$, that is $i_m$. Thus, at the $y$-value just below
the crossing of $U$ and $k'$th strand, we see $a_1$ strands with label
$i_1$, etc. up until $a_m$ strands with label $i_m$, followed by $U$
which also has this label.  Thus, indeed, this slice is associated to $\Bb=(a_1,
a_2,\dots, a_m+1,\dots)$, which is, of course, greater in
lexicographic order than $\Ba$.

Following through the induction on $k$, this shows $(*_\Ba)$ and thus
the desired statement.
 \end{proof}

\begin{thm}\label{rho-surjective}
The homomorphism $\rho\maps \dcurkg\longrightarrow \Tr(\U^\ast_Q(\mf{g}))$, is surjective for all $\mf{g}$.
\end{thm}

\begin{proof}
By the triangular decomposition of Proposition \ref{prop:triangular},
it suffices to show that the image contains any class $[f^\pm]$ or
$[f^0]$ where $f^\pm$, and $f^0$  are in $\Ucat^\pm$, $\Ucat^0$,
respectively.  For $[f^\pm]$, this follows immediately from Theorem
\ref{thm:minus-surjection}.  For $f^0$, this is clear by the
isomorphism between $1_\lambda\dot{\bfU}^0(\g[t])1_\lambda\cong \End(\onel)={\rm Sym}^{|I|}$ for any
weight $\lambda$.
\end{proof}

Note that even if $Q$ is arbitrary (i.e. we may have $v_{ij}\neq
(-1)^{a_{ij}}$) and we weaken our hypotheses to include symmetrizable
Cartan matrices rather than just simply-laced, we
can use the argument of
Theorem \ref{rho-surjective} to show the weaker statement:
\begin{cor}\label{generating-set}
   For any symmetrizable Kac-Moody algebra $\g$ and choice of scalars
  $Q$, the morphisms
   $\mathsf{E}_{i,r}^{(n)}1_{\lambda},\mathsf{F}_{i,r}^{(n)}1_{\lambda}$
   for $i\in I$, $r\ge0$, $n\ge1$ and $\l\in X$,
generate all the morphism spaces of $\Tr(\U^\ast_Q(\mf{g}))$ as a linear category.
\end{cor}

Shan, Varagnolo and Vasserot give a generalization of the current
algebra in \cite[Prop. B.5(a--f)]{SVV}.  There they only show that these
relations hold in the action of
$[\mathsf{E}_{i,r}^{(n)}1_{\lambda}],[\mathsf{F}_{i,r}^{(n)}1_{\lambda}]$
on the trace of the cyclotomic quotient, but the same argument should
show the analogue of Theorem \ref{rho-surjective} in this more general situation.

%
\section{Injectivity results}
\label{sec-injectivity}

\subsection{Cyclotomic quotients}


Fix a highest weight $\lambda$.  The 2-category $\dot{\Ucat}^{\ast}_Q$ has a
principal representation $\dot{\Ucat}^{\ast}_Q(\lambda, *)$ which sends the weight
$\mu$ to the graded category of 1-morphisms $\lambda\to \mu$.

We wish to consider two natural quotients of $\dot{\Ucat}^{\ast}_Q(\lambda, *)$:
\begin{itemize}
\item $\check{\Ucat}^{\lambda,\ast}$ is the quotient of $\dot{\Ucat}^{\ast}_Q(\lambda
  ,*)$ by the subrepresentation generated by
  $\bf{1}_{\lambda+\alpha_i}$ for all $i\in I$.  That is, we set to
  0 any 2-morphism factoring through a 1-morphism of the form $A
  \E_i\onel$ for $A$ arbitrary.
\item ${\Ucat}^{\lambda,\ast}$ is the quotient of
  $\check{\Ucat}^{\lambda,\ast}$ by all positive degree endomorphisms of $\onel$.
\end{itemize}
The categories $\check{\Ucat}^{\lambda,\ast},{\Ucat}^{\lambda,\ast}$ are
automatically idempotent complete, since the degree 0 endomorphisms of
any 1-morphism in $\dot{\Ucat}^{\ast}_Q $ is a finite dimensional algebra
over $\k$, and idempotents in finite dimensional algebras can be
lifted by any surjective map\footnote{We can reduce to the case of
  semi-simple rings, using the fact idempotents modulo the Jacobson radical can
be lifted; for semi-simple rings, this is clear.}.  Such an idempotent
complete category $\mathcal{C}$ can be thought of in terms of the representation
theory of algebras: if we fix an additive generator $X$ of
$\mathcal{C}$ (that is, an object such that every object is a summand
of $X^{\oplus n}$ for some $n$), then the category $\mathcal{C}$ is
equivalent to the finitely generated projective modules over
$A=\End_{\mathcal{C}}(X)$.  By assumption, every object $x$ is the
image of an idempotent $e_x\in \End_{\mathcal{C}}(X^{\oplus n})$, and
the associated projective module is $A^{\oplus n}\cdot e_x$, identifying
$A^{\oplus n}\cong \Hom_{\mathcal{C}}(X, X^{\oplus n})$.

One approach to constructing such a generator is simply to consider
the sum $Y$ of
every monomial in the $\mathcal{E}_i$ and $\mathcal{F}_i$ applied to
$\onel$  (we'll abuse notation and use the same symbol for
$\check{\Ucat}^{\lambda,\ast}$ and ${\Ucat}^{\lambda,\ast}$).
By definition, every object is a summand of a sum of these.  The
resulting endomorphism algebra in ${\Ucat}^{\lambda,\ast}$ is denoted
$DR^{\lambda,\ast}$ in \cite[3.1]{Web5}. The algebra of the same sum of monomials
in $\check{\Ucat}^{\lambda,\ast}$ is obtained in the same way, but only
imposing the relation  \cite[(3.9a)]{Web5}, not \cite[(3.9b-c)]{Web5}
(this algebra is considered in \cite[3.25]{Web5}, but not given a name
or notation).

There is a more ``efficient'' generator, however.  One can easily
confirm using the categorified commutation relation of $\mathcal{E}_i$
and $\mathcal{F}_i$ that just applying monomials in the functors
$\mathcal{F}_i$ will suffice to additively generate the category.  Let $X$ be the sum of all such monomials applied to
$\onel$ (we'll abuse notation as before). Note that the endomorphisms
of these two
additive generators in a given category are Morita equivalent via the bimodule $\Hom(X,Y)$.  The endomorphisms of this generator can also be
described in terms of  the KLR algebra $R=R_Q$.  This has a natural
map $R\to \End_{{\Ucat}^{\lambda,\ast}}(X)$ which is surjective, and a map
$R\to \End_{\check{\Ucat}^{\lambda,\ast}}(X)$, which is close to
surjective; its image together with the bubbles generate these endomorphisms.

Recall that
that $R$ has an algebraic presentation where
$e(\ii)$ are idempotents corresponding to sequences $\ii = (i_1, \dots, i_m)$, and $y_r e(\ii)$ denotes a dot on the $r$th strand.
Let $\lambda$ be a dominant weight, and recall the cyclotomic quotient
$R^{\lambda}$ from \cite{KL1} is defined as the quotient of $R$ by the two sided ideal generated by the relations
\begin{equation}\label{cyc-quotient}
\left\{
 y_1^{ \la i_1, \lambda\ra } e(\ii) =0 \mid \text{for all sequences $\ii$}
\right\}.
\end{equation}
 Likewise, we will also consider
  the deformed cyclotomic quotient $\check{R}^\lambda$ defined in
\cite[Sec. 4.4]{RouQH} and \cite[3.24]{Web5}; this is a quotient of the usual KLR algebra by the
relation
\begin{equation}\label{def-cyc-quotient}
\left(y_1^{\la i_1, \lambda\ra}+q_1^{(i)}y^{\la i_1, \lambda\ra-1}+\dots
+q_{\la i_1, \lambda\ra}^{(i)} \right)e(\mathbf{i})=0
\end{equation}
where each $q_k^{(i)}$ is a free deformation parameter of degree $2k$.
We identify $\k[\{q_k^{(i)}\}]\cong \check{R}^\lambda_\lambda$ with the ring
$\mathbb{A}_\lambda$ by identifying $q_k^{(i)}\mapsto \mathsf{x}_{i,k}$.

Note that the rotation of \eqref{eq_curl} shows that the relations
\eqref{cyc-quotient}  and \eqref{def-cyc-quotient} hold in the
categories ${\Ucat}^{\lambda,\ast}$ and $\check{\Ucat}^{\lambda,\ast}$ (where we
identify $q_k^{(i)}$ with the fake clockwise bubbles with label $i$ of
degree $2k$).  With considerably more effort, one can show that these
relations together with those of $R$ give a presentation of
$\End_{{\Ucat}^{\lambda,\ast}}(X)$ and
$\End_{\check{\Ucat}^{\lambda,\ast}}(X)$ respectively.  In \cite[3.20,
3.25]{Web5}, this result is shown in terms of Morita equivalences
between $R^{\lambda}$ (resp. $\check{R}^\lambda$) and
$\End_{{\Ucat}^{\lambda,\ast}}(Y)$ (resp.
$\End_{\check{\Ucat}^{\lambda,\ast}}(Y)$).

We let $R^{\lambda}_\mu$ (respectively $\check{R}^\lambda_{\mu}$)
 denote the summand of this algebra categorifying
the $\mu$-weight space, that is, that where the labels on strands add
up to $\lambda-\mu$.  The categories of modules over both $R^{\lambda}$ and  $\check{R}^\lambda$ each have a categorical action of $\mf g$, where each $\F_i$ is an induction
functor and $\E_i$ a restriction functor.  From the discussion above,
we have that:

\begin{thm}[\mbox{\cite[6.2]{KK}; \cite[4.25]{RouQH}; \cite[3.20, 3.25]{Web5}}]\label{R-isomorphism}
The category  ${\Ucat}^{\lambda,\ast}(\mu)$
  (resp. $\check{\Ucat}^{\lambda,\ast}(\mu)$) is equivalent to the projective
  modules over the ring ${R}^\lambda_\mu$
  (resp. $\check{R}^\lambda_\mu$).  The Grothendieck groups of
  $\bigoplus_{\mu\leq \lambda}{\Ucat}^{\lambda,\ast}(\mu)$ and $\bigoplus_{\mu\leq \lambda}\check{\Ucat}^{\lambda,\ast}(\mu)$ are
  canonically isomorphic, and both isomorphic to the finite Weyl module $V(\lambda)$.
\end{thm}
On the algebra $R^{\lambda}_\mu$ there is a trace function $\operatorname{tr}_\lambda$
given by  connecting the top and bottom of the diagram
as shown in \cite[3.18]{Web5}, and evaluating this diagram in  the 2-category $\dot{\Ucat}^{\ast}_Q$, as described in
\cite{BHLW}.  This is slightly different from the trace defined first
in \cite{Web5}, since we used a different duality functor on this 2-category, but it agrees with the symmetric modification described
in \cite[Rk. 3.19]{Web5}.

The same definition defines a trace map
$\check{\operatorname{tr}}_\lambda\colon \check{R}^{\lambda}_\mu\to
\mathbb{A}_\lambda$.  This is a Frobenius trace if and only if its
reduction modulo the unique maximal graded ideal of $\mathbb{A}_\lambda$
is, i.e. if $\operatorname{tr}_\lambda$ is.
We'll also require the result:
\begin{thm}[\mbox{\cite[3.18]{Web5}}]\label{frobenius}
   The trace $\operatorname{tr}_\lambda$
   (resp. $\check{\operatorname{tr}}_\lambda$) makes $R^\lambda$ (resp. $\check{R}^{\lambda}$) into a
   symmetric Frobenius algebra over $\k$ (resp. $\mathbb{A}_\lambda$).
\end{thm}

Note that this implies that these categorical modules are integrable
in the sense of Chuang-Rouquier: any object $M$ is killed by
$E_i^m$ or $F_i^m$ for $m\gg 0$, since the resulting object lies in a
trivial weight space.

Then we have a composite of surjections
\[
 \xymatrix{
\dot\curg \ar[r] &
 \Tr(\Ucat(\mf{g})) \ar[r] &
 \Tr(\check{\Ucat}^{\lambda,\ast}(\mf{g})) \ar[r] &
 \Tr(\Ucat^{\lambda,\ast}(\mf{g})).
 }
\]

\begin{prop}\label{Weyl-surject}
There are surjective homogeneous maps of  $\mathbf{U}(\mathfrak{g}[t])$-modules
\[
W(\lambda) \longrightarrow \Tr(\Ucat^{\lambda,\ast}(\mf{g})) \qquad \mathbb{W}(\lambda) \longrightarrow \Tr(\check{\Ucat}^{\lambda,\ast}(\mf{g})).
\]
(More generally, the result holds for any choice of scalars as in Proposition~\ref{prop-cur}.)
\end{prop}
In other types, one can give an analogous definition of a Weyl module, and the same proof as below shows it surjects onto $\Tr(\check{\Ucat}^{\lambda,\ast}(\mf{g}))$.  We have not included any discussion of this case since in contrast to the ADE case, we know almost nothing about its structure.  
\begin{proof}
The trace of $\Ucat^{\lambda,\ast}(\mf{g})$ or
$\check{\Ucat}^{\lambda,\ast}(\mf{g})$ is an integrable representation of the
current algebra $\curg$.  This representation is generated as a
$\curg$-module by the trace of the empty diagram in weight
$\lambda$, which is homogeneous of degree 0.   Furthermore, the cyclotomic relation \eqref{cyc-quotient} or \eqref{def-cyc-quotient} implies that $\mf{n}^+$ acts trivially on this vector.
By the presentation \cite[(3.5)]{CFK}, any integrable $\mf{g}[t]$-module $M$ generated
by an element $m \in M$ satisfying the relations
\begin{equation}
 \mf{n}^+[t]m  =0, \qquad hm=\lambda(h)m\label{global-rels}
\end{equation}
is a quotient of the global Weyl module $\mathbb{W}(\lambda)$.   This
shows the result for $\Tr(\check{\Ucat}^{\lambda,\ast}(\mf{g}))$.  In
$\Ucat^{\lambda,\ast}(\mf{g})$, by definition, all higher degree bubbles
act trivially on $v$.  Thus, the surjection of the global Weyl module factors
through the quotient by the relation $\mf{h}t[t]v=0$.  That is, $V$
is a quotient of the local Weyl module $\Tr(\check{\Ucat}^{\lambda,\ast}(\mf{g}))$.
\end{proof}

%
\subsection{Injectivity of $\rho$}
%

In this section, we will prove an injectivity result.  

\begin{thm}\label{local-isomorphism}
The surjective map $W(\lambda) \longrightarrow \Tr(\Ucat^{\lambda,\ast}(\mf{g}))$ is an isomorphism.
\end{thm}
As mentioned above, we could define an analogous map for $\mf g$ of infinite type.  We have no evidence in either direction for  whether $\rho$ has non-trivial
kernel in this more general case.
\begin{proof}
Let $\lambda_{\rm {min}}$ be the unique minimal dominant weight amongst those
$\leq \lambda$.  If $\lambda$ lies in the root lattice, then $\lambda_{\rm min}=0$; if
$\lambda$ does not lie in the root lattice, then $\lambda_{\rm {min}}$ will be
the unique
highest weight of a minuscule representation in that coset of the root
lattice.  If
$\mf{g}=\mathfrak{sl}_n$, then $V(\lambda_{\rm {min}})=\bigwedge{}^{\!\!k}\mathbb{C}^n$
where $0\leq k <n$ is chosen so that the scalar matrix $e^{2\pi
  i/n}I\in \operatorname{SL}(n)$
acts by $e^{2\pi i k/n}$ on $V(\lambda)$.  For $\mf{g}=\mathfrak{so}_{2n}$,
we have that
\begin{itemize}
\item $V(\lambda_{\rm {min}})=\mathbb{C}$ is the trivial representation if $V(\lambda)$ is a representation of
  $\operatorname{SO}(2n)$ on which $-I\in \operatorname{SO}(2n)$ acts
  trivially,
\item $V(\lambda_{\rm {min}})=\mathbb{C}^{2n}$ is the vector representation if $V({\lambda})$ is a representation of
  $\operatorname{SO}(2n)$ on which $-I\in \operatorname{SO}(2n)$ acts
  by $-1$,
\item $V(\lambda_{\rm {min}})=S^{\pm}$ is one of the two half-spinor representations if
  $V(\lambda)$ is a representation of $\operatorname{Spin}(2n)$ not
  factoring through $\operatorname{SO}(2n)$ (determined by having the
  same action of the center of $\operatorname{Spin}(2n)$ as $V(\lambda)$).
\end{itemize}
  By \cite[3.9]{KN}, the socle filtration of the local Weyl module of
  type $\lambda$ coincides with the degree filtration.
In particular, by \cite[3.14]{KN},
  the socle itself is given by the homogeneous elements of degree $\langle \lambda,\lambda\rangle-\langle
  \lambda_m,\lambda_m\rangle$ and this is a single copy of the simple
  module $V(\lambda_m)$.   By its universal property, any
  graded module $M$ over the current algebra which is generated by a single
  highest weight element of weight $\lambda$ and degree 0 receives a
  surjective map from the local Weyl module.  Thus, if $M$ also
  contains a non-zero element of degree $\langle \lambda,\lambda\rangle-\langle
  \lambda_{\rm {min}},\lambda_{\rm {min}}\rangle$, this map is not zero on the socle of
  the Weyl module.
Since the Weyl module has finite length, every submodule contains a
simple submodule, which lies in the socle by definition.
By \cite[3.8]{KN}, the socle of $W(\lambda)$ is simple, so any non-zero submodule
  of $W(\lambda)$ contains $\operatorname{soc}(W(\lambda))$.  However,
 the kernel of the
  map to $M$ does not contain this submodule, and thus is 0.

By Proposition \ref{Weyl-surject},
  $\Tr(\Ucat^{\lambda,\ast}(\mf{g}))$ is generated by such an element.
  Furthermore, using the isomorphism of Theorem \ref{R-isomorphism},
  the symmetric Frobenius trace $\operatorname{tr}_\lambda$ of Theorem
  \ref{frobenius} on the algebra $R^\lambda_{\lambda_{\rm {min}}}$ induces a non-zero
  functional on $\Tr(\Ucat^{\lambda,\ast}(\mf{g}))$ of degree $-\langle \lambda,\lambda\rangle+\langle
  \lambda_{\rm {min}},\lambda_{\rm {min}}\rangle$, which shows that this space has
  non-zero elements of degree $\langle \lambda,\lambda\rangle-\langle
  \lambda_{\rm {min}},\lambda_{\rm {min}}\rangle$.  Thus, we must have the desired isomorphism.
\end{proof}

The map
$\mathbb{W}(\lambda) \longrightarrow
\Tr(\check{\Ucat}^{\lambda,\ast}(\mf{g}))$ induces a surjective
homogeneous ring homomorphism
$\mathbb{A}_\lambda\to \check{R}^\lambda_\lambda$.
By the definition \cite[3.24]{Web5}, $\check{R}^\lambda_\lambda$ is a polynomial ring
over the fake bubbles.  Thus, these rings have the same Hilbert
series, and this map must be an isomorphism.

\begin{thm}\label{global-isomorphism}
For $\mf{g}$ of type ADE, the surjective map $\mathbb{W}(\lambda) \longrightarrow \Tr(\check{\Ucat}^{\lambda,\ast}(\mf{g}))$ is an isomorphism.
\end{thm}
\begin{proof}
  The induced map is an isomorphism modulo the unique maximal
  homogeneous ideal by Theorem \ref{local-isomorphism}.  At a generic
  maximal ideal, the specialization of $\mathbb{W}(\lambda)$ is
  isomorphic to a tensor product of shifted local Weyl modules for
  fundamental weights, with $\omega_i$ appearing with multiplicity
  $m_i=\la i,\lambda\ra$ by \cite[5.8]{CFK}.  Similarly, the
  specialization of $\check{R}^\lambda$ is Morita equivalent to a
  tensor product of shifted cyclotomic quotients for fundamental
  weights with $R^{\omega_i}$ appearing with multiplicity $m_i$ by
  \cite[2.28]{wKLR}. Thus, the same is true on the level of traces.
  Theorem \ref{local-isomorphism} applied to the fundamental weights
  shows that these modules over
  $\mathbb{A}_\lambda\cong \check{R}^\lambda_\lambda$ have the same
  generic rank.  Thus, any surjective map from one to the other is
  necessarily an isomorphism.
\end{proof}

These results were proven independently in \cite[3.34]{SVV} with
essentially the same proof; these are ultimately an algebraic version
of the argument of \cite[6.1]{KN}.

%
\section{Trace decategorification results}
%

%
\subsection{A trace decategorification of ${\U_Q}(\mf{g})$}
%

Now, we return to considering the trace of the 2-category
${\U_Q}(\mf{g})$ for $\mf{g}$ general (not just ADE type).  The categories ${\Ucat}^{\lambda,\ast}(\mu)$
and $\check{\Ucat}^{\lambda,\ast}(\mu)$ have ``unstarred'' versions
${\Ucat}^{\lambda}(\mu)$ and  $\check{\Ucat}^{\lambda}(\mu)$,
given by the same quotient of ${\U_Q}(\mf{g})$.  These are equivalent
to the graded categories of graded projective $R^\lambda$-modules or $\check{R}^\lambda$-modules
with morphisms being degree 0 module maps.

Taking the trace of these categories results in a somewhat different
structure from the categories we considered earlier: there's a natural 2-functor ${\U_Q}(\mf{g})\to
{\U_Q^*}(\mf{g})$ sending 2-morphisms to degree 0
2-morphisms\footnote{The reader should think of this as analogous to
  the functor from the category of graded modules over a ring with degree
  0 maps as morphisms, to the category with the same objects, but all
  module maps as morphisms. For any 2-morphism $x\colon u\to v$ which is homogeneous of a different degree
$a$, we can compose with a grading shift isomorphism $v\cong v\la
-a\ra$ to obtain a degree 0 2-morphism, which is in the image of this
map.}.

 This
induces a functor on traces $\Tr({\U_Q}(\mf{g}))\to
\Tr({\U_Q^*}(\mf{g}))$.  The image of this functor can only include
degree 0 morphisms, and so misses much of the structure of the trace.  On the other hand, as discussed in Section
\ref{sec:Ustar}, this is a map of $\Z[q,q^{-1}]$ modules with $q-1$
acting trivially on $\Tr({\U_Q^*}(\mf{g}))$, so the behavior of $q$ is
lost in the map.

Recall that in Section \ref{sec:Weyl}, we defined the quantum (finite)
Weyl module $V_q(\l)$; we let $V_q(\l)_\mu$ denote the $\mu$ weight
space of this representation.
Since only a few results in the paper hold in this generality, let us
note that the theorem below holds for all symmetrizable Cartan data,
and choice of scalars $Q$:
\begin{thm} \label{thm:ungraded}
  The Chern character maps
  \begin{align}
h_{\dot{\U}_Q(\mf{g})}&\colon K_0^\k(\dot{\U}_Q(\mf{g})
  )\cong \dot{\bfU}_q(\g)\to \Tr(\U_Q(\mf{g})) \label{eq:8.1}\\
h_{{\Ucat}^{\lambda}(\mu)}&\colon K_0^\k({\Ucat}^{\lambda}(\mu)
  )\cong V_q(\l)_\mu\to \Tr({\Ucat}^{\lambda}(\mu)) \label{eq:8.2}\\
h_{{\check \Ucat}^{\lambda}(\mu)}&\colon K_0^\k({\check \Ucat}^{\lambda}(\mu)
  )\cong V_q(\l)_\mu\to \Tr({\check\Ucat}^{\lambda}(\mu)) \label{eq:8.3}
  \end{align}
are isomorphisms.
\end{thm}

\begin{proof}
  First, we'll show \eqref{eq:8.2}; the argument for \eqref{eq:8.3} is
  identical.  The isomorphism
  $K_0^\k({\Ucat}^{\lambda}(\mu) )\cong V_q(\l)_\mu$ follows from
  \cite[3.21]{Web5} and \cite[6.2]{KK}, after tensoring with $\k$. By Lemma
\ref{lem:Chern-injective},  the map $h_{{\Ucat}^{\lambda}(\mu)}$ is
injective. On the other hand, the argument of Theorem \ref{thm:minus-surjection}
shows that $ \Tr({\Ucat}^{\lambda}(\mu)) $ is spanned by polynomial
endomorphisms.  The only polynomial endomorphisms of degree 0 are the
idempotents $e_\ii$, which lie in the image of $
h_{{\Ucat}^{\lambda}(\mu)}$, so this map is surjective as well.

For \eqref{eq:8.1}, the argument is essentially the same.
We have an isomorphism
$\dot{\bfU}_q({\mf g}) \cong K_0^\k(\dot{\U}_Q({\mf g}))$ as
established in \cite[Th. A]{Webunfurl}.
By Lemma
\ref{lem:Chern-injective},  the map $h_{\dot{\U}_Q(\mf{g})}$ is injective.
On the other
  hand, the elements $\mathsf{E}_{i,0}^{(n)}1_{\lambda},\mathsf{F}_{j,0}^{(n)}1_{\lambda}$
are the only degree 0 elements of
  the generating set  of Corollary
  \ref{generating-set}, with all others of positive degree.  Thus,
  they must generate $\Tr(\U_Q(\mf{g}))$.
The elements
  $\mathsf{E}_{i,0}^{(n)}1_{\lambda},\mathsf{F}_{j,0}^{(n)}1_{\lambda}$
  are given by idempotents, and thus obviously in the image of
  $h_{\dot{\U}_Q(\mf{g})}$.  This shows the map is surjective.
\end{proof}
\begin{rem}
Note that in the case where $\g$ is of finite type (not necessarily
symmetric), we can infer the isomorphism $\dot{\bfU}_q({\mf g}) \cong
K_0^\k(\dot{\U}_Q({\mf g}))$ from \eqref{eq:8.2}.  A surjective map $\dot{\bfU}_q({\mf g}) \to
K_0^\k(\dot{\U}_Q({\mf g}))$ is described in \cite{KL3}, and the
kernel of this map must act trivially on every (finite) Weyl module
since these factor through the map to $K_0^\k(\dot{\U}_Q({\mf g}))$,
which is impossible.  Note that this argument does not work outside
finite type, since highest weight modules must have weights that lie
inside the Tits cone.  Instead, one must consider a categorification
of the tensor product of a highest and a lowest weight module, as in
\cite{Web4,Webunfurl}.  The same arguments as Theorem \ref{thm:ungraded} can
easily be extended to this case as well.
\end{rem}

The special cases of Theorem \ref{thm:ungraded} for
$\mf{g}=\mf{sl}_2,\mf{sl}_3$ are proved in \cite{BHLZ,Marko}, where
the Chern character maps are defined over $\mathbb{Z}$.

\begin{rem}
The notion of a strongly upper-triangular category was defined in \cite[Section 4.1]{BHLZ}.  Such categories possess a distinguished basis of objects $B$ and the results of \cite[Proposition 4.6]{BHLZ} imply that $\Tr(\modC)=\HH_0(\modC) \cong \Bbbk B$ and that all higher Hochschild homology vanishes, $\HH_i(\modC) =0$, for $i>0$.  It follows from results in \cite{Web4} that
the basis of indecomposables in  $\U_Q(\mf{g})$ is strongly
upper-triangular, if $\k$ has characteristic $0$ and $v_{ij}=(-1)^{a_{ij}}$.
Hence, Theorem \ref{thm:ungraded} can be extended to include the fact
that $\HH_i(\U_Q(\mf{g})) = 0$ for $i>0$ under the same hypotheses.
\end{rem}

%
\subsection{A trace decategorification of $\U^\ast_Q(\g)$}
%

 \begin{thm}\label{thm:graded}
Assume $\mf{g}$ is of type ADE, then the linear functor
\begin{equation} 
 \rho \maps \dcurkg\longrightarrow \Tr(\U^\ast_Q(\mf{g})),
\end{equation}
is an isomorphism.
\end{thm}
\begin{proof}
  Lemma~\ref{lem:nokill} implies that the map $\rho$ must be 
  injective, since any element
of its kernel would kill all global Weyl modules.  Combining with
Theorem \ref{rho-surjective} completes the proof.
\end{proof}

%
\section{An action on centers of 2-representations}
%

\subsection{Cyclic 2-categories and the center}

Given a linear 2-category $\CC$ we define the center $Z(\lambda)$ of an object $\lambda \in \Ob( \CC)$ as the commutative ring of endomorphisms $\CC(1_{\lambda}, 1_{\lambda})$, see \cite{GK}.  Note that in the linear 2-category \cat{AdCat} of additive categories, additive functors, and natural transformations, the center  $Z(\C)$ of an object $\C$ is the endomorphism ring of the identity functor $\Id_{\C}$ on $\C$.
Define the {\bf center of objects} of the 2-category $\CC$ as the $Z(\CC) = \bigoplus_{\lambda \in \Ob(\CC)} Z(\lambda)$.

There is a fairly general framework under which the trace of a linear
2-category $\CC$ acts on the center of objects $Z(\mathbf{\cal{K}})$
of any 2-representation $F\maps \CC \to \cal{K}$.  This happens
whenever the 2-category $\CC$ has enough ``coherent'' duality.  This
idea is captured by the notion of a {\em pivotal
  2-category}\footnote{This can be seen as a many object version of a pivotal monoidal category, see \cite{GMV} where traces in this context are studied.  M\"{u}ger points out in \cite[page 11]{Muger} a strict pivotal 2-category can be defined from Mackaay's work~\cite{MackSphere} on spherical 2-categories by ignoring the monoidal structure.
}~\cite{Muger, CKS,Lau1}. In a pivotal 2-category $\CC$ every 1-morphism $x \maps \lambda \to \lambda'$ is equipped with a specified biadjoint morphism $x^* \maps \lambda' \to \lambda$ and 2-morphisms
\begin{align*}
  {\rm ev}_x &\maps x^* x \to 1_{\lambda}    &{\rm coev}_x \maps 1_{\lambda'} \to x x^* \\
  \widetilde{{\rm ev}}_x &\maps x x^* \to 1_{\lambda'}    &\widetilde{{\rm coev}}_x \maps 1_{\lambda} \to x^* x
\end{align*}
satisfying the adjunction axioms.    Then given a 2-morphism $f \maps x \to y$ in $\CC$ we can define the left and right dual of $f$:
\begin{align*}
  & f^* := ({\rm ev}_y \hor \Id_{x^*})(\Id_{y} \hor f \hor \Id_{x^{*}})
  (\Id_{y^*} \hor {\rm coev}_x)\maps y^* \to x^* \nn \\
  & {}^*f := (\Id_{x^*}\hor \widetilde{{\rm ev}}_y)(\Id_{x^*}\hor f \hor \Id_{y^*})
   (\widetilde{{\rm coev}}_x \hor \Id_{y^*}) \maps y^* \to x^*.
\end{align*}
A pivotal 2-category $\CC$ is said to be {\em cyclic (with respect to
  the biadjoint structure)}, or simply a {\em cyclic 2-category},
when the left and right dual agree $f^* = ^*f$, or equivalently $f^{**}=f$.

Let $F\maps \CC \to \mathcal{K}$ be a 2-representation from a cyclic 2-category $\CC$ into a linear 2-category $\cal{K}$.
For $x \maps \lambda \to \lambda'$ a 1-morphism in $\CC$ and  $f \maps x \to x$ a 2-endomorphism in $\CC$ representing a class $[f]$ in $\Tr(\CC)$, then $[f]$ defines an operator $Z(F(\lambda)) \to Z(F(\lambda'))$ sending the element $c \maps 1_{F(\lambda)} \to 1_{F(\lambda)}$ to the element given by the composite
\begin{equation} \label{center-action}
   F(\widetilde{{\rm ev}}_{x})  \ver (F(f)\hor c \hor \Id_{F(x^{\ast})} )\ver F({\rm coev}_{x}) \maps
   1_{F(\lambda')} \to 1_{F(\lambda')} \in Z(F(\lambda')).
\end{equation}
The following proposition is immediate.
\begin{prop}
A 2-representation $F\maps \CC \to \mathcal{K}$ from a cyclic 2-category $\CC$ into a linear 2-category $\cal{K}$ induces an action of $\Tr(\CC)$ on the center of objects $Z(\mathcal{K})$ given by \eqref{center-action}.
\end{prop}

In terms of graphical calculus, an element $c \maps 1_{\lambda} \to 1_{\lambda}$ of $Z(\mathcal{\lambda})$ can be represented by a closed diagram in weight $\lambda$.  A class $[f]$ is represented by a diagram on an annulus with interior region labelled $\lambda$ and exterior region labelled $\lambda'$.  The action of $[f]$ on $c$ is given by forgetting the annulus and placing the diagram for $c$ into the interior region.

\subsection{Cyclicity for the 2-category $\U_Q(\mf{g})$ }

The 2-category $\U_Q(\mf{g})$ has an obvious pivotal structure where
the adjunctions we use are those given by the cups and caps in our
graphical calculus.  The cyclic relations \eqref{eq_almost_cyclic} and  \eqref{eq_cyclic_dot}
implies that this structure is cyclic for an arbitrary choice of $t_{ij}$. Note that the 2-category
defined in \cite{CLau} is cyclic
if and only if  $v_{ij}= 1$ for
all $i,j$.

\begin{prop}\label{c-cyclic}
  For an arbitrary the choice of scalars $t_{ij}$, the pivotal
  structure on $\Ucat_Q(\mf{g})$
   is cyclic.
\end{prop}

\begin{cor}\label{center-weyl}
For arbitrary $\mf{g}$, the category  $\Tr(\Ucat_Q(\mf{g}))$ acts on the center of objects $Z(\mathcal{K})$ in any 2-representation $\Ucat_Q(\mf{g}) \to \cal{K}$.
If $\mf{g}$ is of type ADE and $v_{ij}=-1$, then
$Z(\U^{\lambda,\ast}(\mf{g}))$ and $Z(\check\U^{\lambda,\ast}(\mf{g}))$ can be
identified with the dual local and global Weyl modules, respectively,
using the isomorphism $\Tr(\Ucat_Q(\mf{g})) \cong
\dcurkg$ of Theorem \ref{thm:graded}.
\end{cor}

\begin{proof}
  Only the last statement needs a proof. By Theorem
  \ref{R-isomorphism}, we can instead consider the centers of the
  rings $R^\lambda$ and $\check R^\lambda$.  By Theorem \ref{frobenius},
  the ring $R^\lambda$ is symmetric Frobenius
 and $\check R^\lambda$ is symmetric Frobenius over the base
  ring $\mathbb{A}_\lambda$.
  Under any symmetric Frobenius pairing, the center of a ring and its
  commutator subspace will be mutual orthogonals.

Thus,
  $\operatorname{tr}_\lambda$ induces a duality between
  $W(\lambda)$ and $Z(\U^{\lambda,\ast}(\mf{g}))$ and
  $\check{\operatorname{tr}}_\lambda$ a duality as free modules over
  $\mathbb{A}_\lambda$ between $\mathbb W(\lambda)$ and $Z(\check\U^{\lambda,\ast}(\mf{g}))$. By the definition of the current algebra action on
  the center, this pairing induces a duality of current algebra
  modules, so the result follows.
\end{proof}



\begin{thebibliography}{10}

\bibitem{AAA}
A.A. Adrian.
\newblock Normal division algebras over a modular field.
\newblock {\em Trans. Amer. Math. Soc.}, 36(2):388--394, 1934.

\bibitem{Bag}
M.~Balagovic.
\newblock Degeneration of trigonometric dynamical difference equations for
  quantum loop algebras to trigonometric {C}asimir equations for {Y}angians.
\newblock {\em Comm. Math. Phys.} 334, no. 2, 629–659, 2015.
\newblock \arxiv{1308.2347}.

\bibitem{BGHL}
A.~Beliakova, Z.~Guliyev, K.~Habiro, and A.~Lauda.
\newblock Trace as an alternative decategorification functor.
\newblock {\em Acta Math. Vietnam.}
\newblock to appear. \arxiv{1409.1198 }.

\bibitem{BHLW}
A.~Beliakova, K.~Habiro, A.~Lauda, and B.~Webster
\newblock Cyclicity for categorified quantum groups.
\newblock {\em J. Algebra} 452, 118–132, 2016. 
\newblock \arxiv{1506.04671}.

\bibitem{BHLZ}
A.~Beliakova, K.~Habiro, A.~Lauda, and M.~Zivkovic.
\newblock Trace decategorification of the categorified quantum $\mathfrak{sl}(2)$.
\newblock  {\em Math. Ann.} 367 (2017), no. 1--2, 397--440, 2014.
\newblock \arxiv{1404.1806}.

\bibitem{Bor}
F.~Borceux.
\newblock {\em Handbook of categorical algebra. 1}, volume~50 of {\em
  Encyclopedia of Mathematics and its Applications}.
\newblock Cambridge University Press, Cambridge, 1994.

\bibitem{Brundandef}
\newblock J.~Brundan,
\newblock On the definition of {K}ac-{M}oody 2-category,
\newblock {\em Math. Ann.} 364, no. 1--2, 353--372, 2016.
\newblock  \arxiv{1501.00350}.

\bibitem{Brundan}
J.~Brundan.
\newblock Symmetric functions, parabolic category {$\cal{O}$}, and the
  {S}pringer fiber.
\newblock {\em Duke Math. J.}, 143(1):41--79, 2008.
\newblock \arxiv{math/0608235}.

\bibitem{BK1}
J.~Brundan and A.~Kleshchev.
\newblock Blocks of cyclotomic {H}ecke algebras and {K}hovanov-{L}auda
  algebras.
\newblock {\em Invent. Math.}, 178(3):451--484, 2009.
\newblock \arxiv{0808.2032}.

\bibitem{BS}
J.~Brundan and C.~Stroppel.
\newblock Highest weight categories arising from {K}hovanov's diagram algebra
  {III}: category {O}.
\newblock {\em Rep. Theory}, 15:170--243, 2011.
\newblock \arxiv{0812.1090}.

\bibitem{CW}
A.~C{\u{a}}ld{\u{a}}raru and S.~Willerton.
\newblock The {M}ukai pairing. {I}. {A} categorical approach.
\newblock {\em New York J. Math.}, 16:61--98, 2010.
\newblock \arxiv{0707.2052}.

\bibitem{Cautis-rigid}
S.~Cautis.
\newblock Rigidity in higher representation theory.
\newblock 2014.
\newblock \arxiv{1409.0827}.

\bibitem{CKLquiver}
S.~Cautis, J.~Kamnitzer, and A.~Licata.
\newblock Coherent sheaves on quiver varieties and categorification.
\newblock {\em Math. Ann.} 357, no. 3, 805--854, 2013.
\newblock \arxiv{1104.0352}.

\bibitem{CLau}
S.~Cautis and A.~D. Lauda.
\newblock Implicit structure in 2-representations of quantum groups.
\newblock {\em Selecta Mathematica}, pages 1--44, 2014.
\newblock \arxiv{1111.1431}.

\bibitem{CauLic}
S.~Cautis and A.~Licata.
\newblock Vertex operators and 2-representations of quantum affine algebras.
\newblock 2011.
\newblock \arxiv{1112.6189}.

\bibitem{CFK}
V.~Chari, G.~Fourier, and K.~Tanusree.
\newblock A categorical approach to {W}eyl modules.
\newblock {\em Transform. Groups}, 15(3):517--549, 2010.

\bibitem{CL}
V.~Chari and S.~Loktev.
\newblock Weyl, Demazure and fusion modules for the current algebra of
  $\fsl_{r+1}$.
\newblock {\em Adv. Math.}, 207:928--960, 2006.
\newblock \arxiv{RT/0608235}.

\bibitem{CR}
J.~Chuang and R.~Rouquier.
\newblock Derived equivalences for symmetric groups and
  {$\mf{sl}_2$}-categorification.
\newblock {\em Ann. of Math. (2)}, 167(1):245--298, 2008.

\bibitem{CKS}
J.~R.~B. Cockett, J.~Koslowski, and R.~A.~G. Seely.
\newblock Introduction to linear bicategories.
\newblock {\em Math. Structures Comput. Sci.}, 10(2):165--203, 2000.
\newblock The Lambek Festschrift: mathematical structures in computer science
  (Montreal, QC, 1997).

\bibitem{D}
 V.~G.~Drinfeld.
\newblock A new realization of Yangians and of quantum affine
algebras.
\newblock {\em Soviet Math. Dokl.} \textbf{36} (1988), no. 2, 212--216

\bibitem{GK}
N.~Ganter and M.~Kapranov.
\newblock Representation and character theory in 2-categories.
\newblock {\em Adv. Math.}, 217(5):2268--2300, 2008.
\newblock \arxiv{math/0602510}.

\bibitem{Gar}
H.~Garland.
\newblock The arithmetic theory of loop algebras.
\newblock {\em J. Algebra}, 53(2):480--551, 1978.

\bibitem{GMV}
N.~Geer, B.~Patureau-Mirand, and A.~Virelizier.
\newblock Traces on ideals in pivotal categories.
\newblock {\em Quantum Topol.}, 4(1):91--124, 2013.

\bibitem{HS}
D.~Hill and J.~Sussan.
\newblock The {K}hovanov-{L}auda 2-category and categorifications of a level
  two quantum {$\mf{sl}_n$} representation.
\newblock {\em Int. J. Math. Math. Sci.}, pages Art. ID 892387, 34, 2010.

\bibitem{KK}
S.-J. Kang and M.~Kashiwara.
\newblock Categorification of highest weight modules via
  {K}hovanov-{L}auda-{R}ouquier algebras.
\newblock  {\em  Invent. Math.} 190, no. 3, 699--742, 2012.
\newblock \arxiv{1102.4677}.

\bibitem{Kelly}
G.~M. Kelly.
\newblock On the radical of a category.
\newblock {\em J. Austral. Math. Soc.}, 4:299--307, 1964.

\bibitem{KL1}
M.~Khovanov and A.~Lauda.
\newblock A diagrammatic approach to categorification of quantum groups {I}.
\newblock {\em Represent. Theory}, 13:309--347, 2009.
\newblock \arxiv{0803.4121}.


\bibitem{KL2}
M.~Khovanov and A.~Lauda.
\newblock A diagrammatic approach to categorification of quantum groups {II}.
\newblock {\em Trans. Amer. Math. Soc.} 363, 2685--2700, 2011.
\newblock \arxiv{0804.2080}.

\bibitem{KL3}
M.~Khovanov and A.~Lauda.
\newblock A diagrammatic approach to categorification of quantum groups {III}.
\newblock {\em Quantum Topology}, 1:1--92, 2010.
\newblock \arxiv{0807.3250}.

\bibitem{KN}
R.~Kodera and K.~Naoi.
\newblock {L}oewy series of Weyl modules and the Poincar\'e polynomials of quiver
  varieties, 2011.
  \newblock {\em Publ. Res. Inst. Math. Sci.} 48, no. 3, 477–500, 2012.
\newblock \arxiv{1103.4207}.

\bibitem{Lau1}
A.~D. Lauda.
\newblock A categorification of quantum sl(2).
\newblock {\em Adv. Math.}, 225:3327--3424, 2008.
\newblock \arxiv{0803.3652}.

\bibitem{Lau3}
A.~D. Lauda.
\newblock An introduction to diagrammatic algebra and categorified quantum
  ${\mathfrak{sl}}_2$.
\newblock {\em Bulletin Inst. Math. Academia Sinica}, 7:165--270, 2012.
\newblock \arxiv{1106.2128}.

\bibitem{LQR}
A.D. Lauda, H.~Queffelec, and D.~Rose.
\newblock Khovanov homology is a skew Howe 2-representation of categorified
  quantum $\mathfrak{sl}(m)$.
\newblock 
{\em Algebr. Geom. Topol.} 15, no. 5, 2517--2608, 2015.
\newblock \arxiv{1212.6076}.

\bibitem{McD}
I.~G. Macdonald.
\newblock {\em Symmetric functions and {H}all polynomials}.
\newblock The Clarendon Press Oxford University Press, New York, 1979.
\newblock Oxford Mathematical Monographs.

\bibitem{MackSphere}
M.~Mackaay.
\newblock Spherical {$2$}-categories and {$4$}-manifold invariants.
\newblock {\em Adv. Math.}, 143(2):288--348, 1999.
\newblock \arxiv{math/9805030}.

\bibitem{Mitchell}
B.~Mitchell.
\newblock Rings with several objects.
\newblock {\em Advances in Math.}, 8:1--161, 1972.

\bibitem{Mackaay-foams}
M.Mackaay.
\newblock sl(3)-foams and the {K}hovanov-{L}auda categorification of quantum
  sl(k).
\newblock 2009.
\newblock \arxiv{0905.2059}.

\bibitem{MY}
M.Mackaay and Y.~Yonezawa.
\newblock $\mathfrak{sl}({N})$-web categories and categorified skew Howe duality
\newblock {\em J. Pure Appl. Algebra} 223, no. 5, 2173–2229, 2019.

\bibitem{Muger}
M.~M{\"u}ger.
\newblock From subfactors to categories and topology. {I}. {F}robenius algebras
  in and {M}orita equivalence of tensor categories.
\newblock {\em J. Pure Appl. Algebra}, 180(1-2):81--157, 2003.
\newblock \arxiv{0111204}.

\bibitem{Nak98}
H.~Nakajima.
\newblock Quiver varieties and {K}ac-{M}oody algebras.
\newblock {\em Duke Math. J.}, 91(3):515--560, 1998.

\bibitem{QR}
H.~Queffelec and D.~Rose.
\newblock The $\mathfrak{sl}_n$ foam 2-category: a combinatorial formulation of
  {K}hovanov-{R}ozansky homology via categorical skew {H}owe duality.
\newblock {\em Adv. Math.} 302, 1251–1339, 2016.

\bibitem{RouQH}
R.~Rouquier.
\newblock Quiver {H}ecke algebras and 2-{L}ie algebras.
\newblock {\em Algebra Colloq.} 19, no. 2, 359–410, 2012. 

\bibitem{Rou2}
R.~Rouquier.
\newblock 2-{K}ac-{M}oody algebras, 2008.
\newblock \arxiv{0812.5023}.

\bibitem{ShanCrystal}
P.~Shan.
\newblock Crystals of {F}ock spaces and cyclotomic rational double affine
  {H}ecke algebras.
\newblock {\em Ann. Sci. \'Ec. Norm. Sup\'er. (4)}, 44(1):147--182, 2011.

\bibitem{SVV}
P.~Shan, M.~Varagnolo, and E.~Vasserot.
\newblock On the center of {Q}uiver-{H}ecke algebras.
\newblock {\em Duke Math. J.} 166, no. 6, 1005--1101, 2017. 

\bibitem{VV}
M.~Varagnolo and E.~Vasserot.
\newblock Canonical bases and {KLR}-algebras.
\newblock {\em J. Reine Angew. Math.}, 659:67--100, 2011.

\bibitem{Webcatq}
B.~Webster.
\newblock A categorical action on quantized quiver varieties.
\newblock {\em Math. Z.} 292, no. 1--2, 611--639, 2019. 

\bibitem{wKLR}
B.~Webster.
\newblock Weighted {K}hovanov-{L}auda-{R}ouquier algebras.
\newblock {\em Doc. Math.}
  24, 209--250, 2019. 
  
\bibitem{Web4}
B.~Webster.
\newblock Canonical bases and higher representation theory.
\newblock {\em Compositio Mathematica} 151, no. 1, 121--166, 2015.

\bibitem{Web5}
B.~Webster.
\newblock Knot invariants and higher representation theory.
\newblock {\em Mem. Amer.
  Math. Soc.} \textbf{250}, no.~1191, 141 pp., 2017.
  
\bibitem{Webunfurl}
B.~Webster.
\newblock Unfurling Khovanov-Lauda-Rouquier algebras.
\newblock 2016.
\newblock \arxiv{1603.06311}.

\bibitem{Marko}
M.~Zivkovic.
\newblock Trace decategorification of the quantum $\mf{sl}_3$, 2014.
\newblock {\em Appl. Categ. Structures} 31, no. 1, 9, 2023. 
\newblock \arxiv{1405.2314}.

\end{thebibliography}

%

\end{document}